\documentclass[a4paper,12pt]{amsart}

\usepackage[latin1]{inputenc}
\usepackage[english]{babel}
\usepackage{amssymb,amsmath}
\usepackage{hyperref}
\usepackage{graphicx}
\usepackage{pifont,bbding}
\usepackage{amsfonts}
\usepackage{amsthm}
\usepackage{latexsym}
\usepackage{comment}

\newcommand{\R}{{\mathbb R}}

\newcommand{\N}{{\mathbb N}}

\newcommand{\de}{\partial}
\newcommand{\al}{\alpha}
\newcommand{\be}{\beta}
\newcommand{\pa}{\partial}
\newcommand{\erm}{{\mathrm{e}}}

\renewcommand{\rho}{\varrho}
\renewcommand{\epsilon}{\varepsilon}
\renewcommand{\theta}{\vartheta}

\newcommand{\la}{\lambda}

\renewcommand{\d}{\delta}
\renewcommand{\a}{\alpha}
\renewcommand{\b}{\beta}
\newcommand{\ga}{\gamma}

\newcommand{\Rn}{{\mathbb{R}}^{n}}

\newcommand{\Rm}{{\mathbb{R}}^{m}}
\newcommand{\I}{{\mathcal{I}}}
\newcommand{\A}{\mathcal{A}}

\newcommand{\D}{{\mathcal{D}}}

\newcommand{\ex}{{\mathrm{e}}}

\newcommand{\barint}
{\rule[.036in]{.12in}{.009in}\kern-.16in \displaystyle\int}

\theoremstyle{plain}
\newtheorem{teo}{Theorem}[section]
\newtheorem{prop}[teo]{Proposition}
\newtheorem{defi}[teo]{Definition}
\newtheorem{lemma}[teo]{Lemma}
\newtheorem{corol}[teo]{Corollary}

\theoremstyle{remark}

\newtheorem{remark}[teo]{Remark}

\renewcommand{\prec}{\preceq}

\begin{document}

\title[Extremal curves in nilpotent Lie groups]{Extremal curves in nilpotent
Lie groups}

\author[Le Donne]{Enrico Le Donne}

\address[Le Donne]{Department of Mathematics and Statistics, P.O. Box 35,
FIN-40014,
University of Jyv\"askyl\"a, Finland}%
\email{ledonne@msri.org}

\author[Leonardi]{Gian Paolo
Leonardi}
\address[Leonardi]{Universit\`a di Modena e Reggio Emilia,
  Dipartimento di Scienze
Fisiche, Informatiche e Matematiche, via Campi 213/b, 41100
  Modena, Italy}

\email{gianpaolo.leonardi@unimore.it}

\author[Monti]{Roberto Monti}
\address[Monti and Vittone]{Universit\`a di Padova, Dipartimento di Matematica,
via Trieste 63, 35121 Padova, Italy}
\email{monti@math.unipd.it}

\author[Vittone]{Davide Vittone}
%\address[Vittone]{Universit\`a di Padova, Dipartimento di Matematica Pura ed
%Applicata, via Trieste 63, 35121 Padova, Italy}
\email{vittone@math.unipd.it}

\date{July 17, 2012}


\renewcommand{\subjclassname}{%
 \textup{2010} Mathematics Subject Classification}

\thanks{This work was partially supported by the Fondazione CaRiPaRo Project
``Nonlinear Partial Differential Equations: models, analysis, and
control-theoretic problems'', Padova.}

\begin{abstract}
We classify extremal curves in free nilpotent Lie groups.
The classification is obtained via an explicit integration of the adjoint
equation in Pontryagin Maximum Principle. It turns out that abnormal extremals
are precisely the horizontal curves contained in algebraic
varieties  of a specific type. We also extend the results to the nonfree case.

\end{abstract}

\maketitle
\setcounter{tocdepth}{2}
%\tableofcontents

\section{Introduction}

Let $M$ be a differentiable manifold and $\mathcal D\subset TM$ a
bracket generating distribution. A Lipschitz curve $\gamma:[0,1]\to M$ is
{\em horizontal} if $\dot\gamma(t) \in \mathcal D(\gamma(t))$ for
a.e.~$t\in[0,1]$.
Fixing a quadratic form on
$\mathcal D$, one can define the length of horizontal curves.
A distance  on $M$ can be defined by minimizing the length
of horizontal curves connecting pair of points.  This is known as sub-Riemannian
or
Carnot-Carath\'eodory distance.
If the resulting metric space is proper,
length minimizing curves between any given pair of points do exist.

The main open problem in the field is the regularity of 
length minimizing curves, see \cite[Problem 10.1]{montgomery}. Length minimizing
curves may in fact be {\em abnormal
extremals} in the sense of Geometric Control Theory: while normal extremals
are always smooth, abnormal ones are apriori
only Lipschitz continuous. We refer the reader to Section
\ref{TWO-TWO} for precise definitions. Let us only recall here that a
horizontal curve is an abnormal extremal if the
end-point mapping is singular at this curve, i.e., its differential   is not
surjective. The notion of abnormal extremal is also related to that of rigid
curve, see \cite{Bryant-Hsu,ZZ}.  
Abnormal extremals depend only on the structure $(M,\mathcal D)$
and not on the fixed quadratic form on $\mathcal D$;
it is well known that they can be length minimizing, see \cite{montg1,LS}.
On the other hand, in some structures the presence of singularities prevent  an
abnormal
extremal to be  length minimizing
\cite{LM,monti}. To our best knowledge, no example of minimizing nonsmooth
horizontal curve
is presently known.

The set of   abnormal
extremals of a structure $(M,\mathcal D)$ is an  interesting object that is
not yet well-understood. See, however, the deep second order analysis
\cite{AS2}.
In this paper, we describe this set when $M$ is a (connected, simply
connected) free nilpotent Lie group and $\mathcal D$ is the left-invariant
subbundle spanned by a system of generators of its Lie
algebra. We obtain a description that is constructive and of a
purely algebraic type. It is constructive in the sense that it permits to
construct  
abnormal extremals of any (feasible) desired type: Goh extremals, extremals with
given
corank, extremals of minimal order in the sense of \cite{CJT}, etc. The
classification is purely algebraic in the sense that it only depends
on the structure constants of the algebra of the group. 
These results also extend to other nilpotent groups. In particular, we consider
connected, simply connected, nilpotent and stratified Lie groups (Carnot
groups): this is of special interest because, by  Mitchell's
theorem, 
Carnot groups are the infinitesimal models
of equiregular sub-Riemannian structures.

\medskip

Let us give some flavour of the results contained in the paper. We start by
fixing a basis $X_1,\ldots,X_n$ of an $n$-dimensional free nilpotent  Lie
algebra $\mathfrak g$ generated by $r$ elements $X_1,\dots, X_r$; the choice of
this basis has to be done according to a very precise algorithm due to M. Hall
Jr., see \cite{Hall}. A  detailed description of this algorithm is contained in
Section \ref{HALLE}. 
This basis determines a collection of \emph{generalized structure constants},
see \eqref{GSC}. Using these constants, 
for any
$i=1,\ldots,n$ and for any 
multi-index $\alpha\in \I := \N^n=\{0,1,2,\dots\}^n$, 
we define certain linear mappings 
$\phi_{i\al}:\Rn\to\R$,  see \eqref{LinM}. When $|\a|:=\alpha_1+\dots+\alpha_n$
is large we have
$\phi_{i\al}=0$ because of the nilpotency.
For each $i=1,\dots,n $ and $v\in\R^n$, we introduce the 
polynomials
$P_i^v:\Rn\to\R$
\begin{equation}\label{gestalt-intro}
P_i^v(x)= \sum_{\a\in \I} \phi_{i\alpha}(v) 
 x^\a,\quad x\in\Rn,
\end{equation}
where we let $x^\alpha=x_1^{\alpha_1}\cdots x_n^{\alpha_n}$. 
These polynomials  enjoy some remarkable properties that are
discussed in Proposition \ref{propderivpolinomi}.

A result due to M.~Grayson and R.~Grossman \cite{GrGr} ensures that $\mathfrak
g$ is isomorphic to a certain algebra of vector fields in $\R^n$; it turns out
that $G$ can be identified with
$\Rn$ via
exponential coordinates of the second type associated with  the vector fields
corresponding to  $X_1,\ldots,X_n$, see Proposition \ref{EXPO}. For any
$v\in\Rn$, we call the set   
\[
  Z_v = \big\{x\in  \Rn: P_1^v(x) =\ldots=P_r^v(x) = 0\big\}
\]
an \emph{abnormal  variety} of  $G$  of {\em corank} $1$; when $v\neq 0$ we
have
$Z_v\neq\Rn$, see Proposition \ref{lemma-0}.

One of the   results proved in the paper is the following theorem; see
Definition \ref{CORANK} for
the notion of corank.

\begin{teo}\label{cor:polin:33} 
Let $G=\Rn$ be a free nilpotent Lie group and let $\ga:[0,1]\to G$
be a  horizontal curve with
$\ga(0)=0$. The following statements are
equivalent:

\begin{itemize}
 \item[A)] The curve $\ga$ is an abnormal extremal of corank $m\geq 1$.

 \item[B)] There exist $m$ linearly independent vectors
$v_1,\ldots,v_m\in\Rn$ such that $\gamma(t)\in  Z_{v_1}\cap\ldots\cap
Z_{v_m}$ for all $t\in[0,1]$.
\end{itemize}
\end{teo}

A stronger version of Theorem \ref{cor:polin:33} holds when $\gamma$ is a
{\em strictly abnormal} minimizer. In this case,  the Goh
condition   (see
Theorem \ref{teo:Goh}) implies that for all $t\in[0,1]$ we have 
\begin{equation}\label{gohvar-intro}
\gamma(t)\in \big\{x\in  \Rn: P_1^v(x) =\ldots=P_{s}^v(x) =
0\big\},
\end{equation}
where $s=\mathrm{dim}(\mathfrak g_1\oplus [\mathfrak g_1,\mathfrak g_1])$ and
$\mathfrak g_1$ is the first layer of $\mathfrak g$, i.e., the subspace spanned
by the
generators of $\mathfrak g$. We call the algebraic sets as in
\eqref{gohvar-intro} {\em Goh varieties} of $G$.

The  proof of Theorem \ref{cor:polin:33} relies upon Theorem \ref{teointduale},
where, for any extremal (normal or abnormal), we 
explicitly compute the dual curve  provided by Pontryagin Maximum
Principle. We started from the usual version of the differential equation for
the dual curve  in Optimal Control Theory,
see Theorem \ref{teo:eqdualvar1}, formula \eqref{ADJ}. The coordinates
$\la=(\la_1,\ldots,\la_n)$ of the dual curve in the basis of $1$-forms dual to
the fixed frame of vector fields, see \eqref{omegalambda}, satisfy the
system of equations
\eqref{8}. 
After studying some Lie groups of step 4, 5 and 6, we could  guess the
general formula for  $\lambda$.  This led us to the
polynomials in \eqref{gestalt-intro}: indeed, in Theorem \ref{teointduale}, we
prove that
\begin{equation}\label{eqlambda-intro}
\la_i(t) = P_i^v(\gamma(t)), \quad \text{for all $t\in[0,1]$ and $i=1,\dots,n$},
\end{equation}
where  $v:=\lambda(0)\neq 0$ is the initial condition.
To check formulae 
\eqref{eqlambda-intro},  we had to
understand how the
Jacobi identity enters the process of integration of the adjoint equation. The
key technical tool is a
combinatorial identity for iterated commutators in 
free nilpotent Lie algebras that  is proved by induction in Lemma
\ref{Roberto}. We were not able to find it in the Lie algebra
literature.

Formulae  \eqref{eqlambda-intro}
provide an explicit integration of the second equation in the
Hamiltonian system \eqref{HJ}, see Remark \ref{mamax}. These
formulae hold for both normal and abnormal extremals and, from the technical
viewpoint, are the main result of the paper.

In Section \ref{sec:notfree}, we extend our results from the free case to the
case of nonfree Carnot groups. The results here are less explicit
because they involve a ``lifting'' procedure from a nonfree group to a free one,
see Theorem \ref{PIV}.
However, the  results are precise enough to deduce in a purely algebraic way
some interesting facts on Carnot-Carath\'eodory geodesics, such as the
$C^\infty$
smoothness of length minimizing curves in Carnot groups of step 3 (see
\cite{cinesi}), and
Gol\`e-Karidi's example \cite{GK} of a strictly abnormal extremal in a Carnot
group. These and other examples are
briefly discussed in  Section \ref{EXA}.
\medskip

The results of this paper raise several   questions.
As noticed, abnormal extremals are precisely horizontal curves lying inside
abnormal
varieties: consequently, the interplay between the horizontal distribution
and the tangent space to the
variety determines the structure of all possible singularities of abnormal
extremals. The precise knowledge of these singularities could permit a fine
analysis of the length minimality properties of abnormal extremals.
It would be interesting to use the
techniques developed in \cite{LM} and \cite{monti} to exclude corners and other
kinds of singularities for length minimizers. See, however, the nonsmooth
extremal of Section \ref{Example2}.

Also, it would be   interesting to understand whether our results can be
extended to wider classes of Lie groups  or to other families of sub-Riemannian
structures. For instance, one could wonder whether abnormal extremals of
analytic sub-Riemannian manifolds (namely, analytic manifolds with analytic
horizontal distribution) are contained in an analytic variety and whether such 
varieties  can be characterized in some way.

As noted above, the family of abnormal extremals is a geometric object
associated with a manifold $M$ along with a bracket generating distribution
$\mathcal D \subset TM$. In the same spirit, an abnormal variety  is an
intrinsic algebraic subset of a (free) Lie group. 
The study of abnormal varieties could be  of  interest in real
algebraic
geometry. Na\-tu\-ral questions concern generic dimension, smoothness,
structure of singularities and  topology of abnormal varieties.
In Section  \ref{Example2}, we list the quadrics defining
Goh extremals
in
the free nilpotent Lie group   of rank $3$ and
step $4$.

%%%%%%%%%%%%%%%%%%%%%%%%%%%%%
%%%%%%%%%%%%%%%%%%%%%%%%%%%%%

\section{Extremal curves in sub-Riemannian geometry}

\label{TWO-TWO}
In this section, we recall some basic facts concerning extremal curves in
sub-Riemannian geometry. Let $X_1,\ldots,X_r$, $r\geq 2$, be linearly
independent smooth vector fields in
$\Rn$, $n\geq3$, and let $\mathcal D$ denote the distribution of $r$-planes 
$\mathcal D(x)
= \mathrm{span}\{X_1(x) , \ldots,X_r(x) \}$ with $x\in \Rn$.
$\mathcal D$ is called \emph{horizontal distribution} and its sections are
called horizontal vector fields. 
With respect to the standard basis $\tfrac{\partial}{\partial
x_1},\ldots,\tfrac{\partial}{\partial x_n}$ of $\Rn$ we
have for any $j=1,\dots,r$
\begin{equation}\label{COFFI}
X_j =\sum_{i=1}^n X_{ji}  \frac{\partial}{\partial x_i},
\end{equation}
where $X_{ji}:\Rn\to\R$ are  smooth   functions.

A Lipschitz curve
$\ga:[0,1]\to\Rn$ is $\mathcal D$-\emph{horizontal}, or simply
\emph{horizontal}, if there exists
a vector of functions $h=
(h_1,\ldots,h_r)\in  
L^\infty([0,1];\R^r)$ such that
\[
\dot\ga = %h\cdot X(\ga) := 
 \sum_{j=1}^r h_j X_j(\ga),\quad \text{a.e. on }[0,1].
\]
The functions $h$ are called controls of $\ga$.
Let $g_x$ be the quadratic form on $\D(x)$ 
making $X_1,\ldots,X_r$ orthonormal.
The
horizontal length of a horizontal curve $\ga$ is then 
\[
 L(\gamma) = \Big( \int_0^1 g_{\ga(t)}(\dot\gamma(t) )dt\Big)^{1/2} 
=\Big( \int_0^1|h(t)|^2 dt\Big)^{1/2} .
\]
Here, we are adopting the $L^2$ definition for the length.
For any couple of points $x,y\in\Rn$, we can define the distance 
\begin{equation}\label{dist}
 d(x,y) = \inf\Big\{ L(\ga): \text{$\ga$ is  horizontal, $\ga(0)=x$
and $\ga(1)=y$} \Big\}.
\end{equation}
If the above set is nonempty for any $x,y\in\Rn$, then $d$ is a distance on
$\Rn$, usually called Carnot-Carath\'eodory distance.
This holds when the vector fields $X_1,\ldots,X_r$ satisfy the H\"ormander
bracket-generating condition: at any point of $\Rn$,
 $X_1,\ldots,X_r$ along with their commutators of sufficiently
large length span  a vector space of full dimension $n$.

If the resulting metric space $(\Rn,d)$ is  complete, the
infimum in \eqref{dist} is attained. We call a curve $\ga$ providing the minimum
a \emph{(length) minimizer}. If $h$ is the vector of the controls of a
minimizer $\ga$, we
call the pair $(\ga,h)$ an \emph{optimal pair}. The minimizer, which in general
is
not
unique,
 is found within the class of Lipschitz curves, which are differentiable almost
everywhere.  
  Pontryagin Maximum Principle
provides necessary conditions for a horizontal curve to be a minimizer.

\begin{teo}\label{teo:eqdualvar1} Let $(\ga,h)$ be an optimal pair. 
Then there exist $\xi_0\in\{0,1\}$ and a Lipschitz curve
$\xi:[0,1]\to\Rn$
such that:
\begin{itemize}
 \item[i)] $\xi_0+|\xi|\neq 0$ on $[0,1]$;

 \item[ii)] $\xi_0 h_j +\langle \xi, X_j(\ga)\rangle =0$ on
$[0,1]$ for all $j=1,\ldots, r$; 

 \item[iii)]  the coordinates $\xi_k$, $k=1,\ldots,n$, of the curve  
$\xi$ solve  the system of differential equations
\begin{equation}\label{ADJ}
    \dot\xi_k = -   \sum_{j=1}^ r \sum_{i=1}^n \frac{\de{X_{ji}}}{\de
x_k}(\gamma)
h_j \xi_i,  \quad\text{a.e.~on $[0,1]$.}
\end{equation}
\end{itemize}
\end{teo}

\noindent 
We refer to \cite[Chapter 12]{AS} for a proof of Theorem
\ref{teo:eqdualvar1} in a more general framework.
Equations \eqref{ADJ} are called  \emph{adjoint equations}.

\begin{defi}   
We say that a   horizontal curve
$\ga:[0,1]\to\Rn$  is an \emph{extremal} if there exist
$\xi_0\in\{0,1\}$   
and 
$\xi\in\mathrm{Lip}([0,1];\Rn)$
such that i), ii), and iii) in Theorem \ref{teo:eqdualvar1} hold.

We say that $\ga$ is a \emph{normal extremal} if there exists such a pair
$(\xi_0,\xi)$  with $\xi_0\neq 0$.  

We say that $\ga$ is an \emph{abnormal extremal}
if there exists such a pair with $\xi_0=0$. 

We say that $\ga$ is a
\emph{strictly abnormal extremal} if $\ga$ is an abnormal extremal but not a
normal one.
 \end{defi}

\noindent
We call  a curve $\xi$ satisfying i), ii), and iii) in Theorem
\ref{teo:eqdualvar1} a \emph{dual curve}
of $\gamma$.  We identify the curve $\xi$ with the curve of $1$-forms in $\Rn$
\[
 \xi = \xi_1 dx_1+\ldots+ \xi_n dx_n.
\]
The dual curve $\xi$ is constructed in the following way.
Let $\Phi:[0,1]\times \Rn\to \Rn$ be the flow in $\Rn$ associated with the
  controls $h_1,\ldots,h_r\in
L^2([0,1])$ of $\gamma$.
Namely, let $\Phi(t,x)=\gamma_x(t)$ where $\gamma_x:[0,1]\to \R^n$ is the
solution to the problem
\[
 \dot\gamma_x =\sum_{j=1}^r h_j X_j(\gamma_x)\quad \text{a.e.~and
$\gamma_x(0)=x\in \R^n$}.
\] 
We also let $\Phi_t(x) = \Phi(t,x)$. 
Then, dual curves $\xi:[0,1]\to T^*\Rn$ are of the form
\begin{equation}
\label{popix}
 \xi (t) = (\Phi_t^{-1})^*\xi(0),\quad t\in[0,1]
\end{equation}
for suitable $\xi(0)\neq0$. Above, $(\Phi_t^{-1})^*\xi(0)$
denotes the 
 pull-back
of $1$-forms by the diffeomorphism $\Phi_t^{-1}:\Rn\to \Rn$.
We refer to \cite{AS} for this characterization of dual curves.

Normal extremals are smooth curves. In fact, by ii) the controls satisfy 
\begin{equation}\label{acca}
h_j = - \langle \xi,X_j(\ga) \rangle\quad \text{a.e.~on }[0,1]
\end{equation}
for any $j=1,\dots,r$. This along with the adjoint equation \eqref{ADJ} implies
the $C^\infty$-smoothness of $\ga$. 
Moreover, the pair $(\gamma,\xi)$ solves the system of Hamilton's equations
\begin{equation}\label{HJ}
 \dot\gamma = \frac{\partial  H}{\partial \xi} (\gamma,\xi),
  \quad 
 \dot\xi = - \frac{\partial H}{\partial x} (\gamma,\xi),
\end{equation}
where $H$ is the Hamiltonian function
\[
  H(x,\xi)=-\frac 12 \sum_{j=1}^r \langle X_j(x),
\xi\rangle^2.
\]

For abnormal extremals we have, for any
$j=1,\ldots,r$,
\begin{equation}\label{condnecessabnormali}
\langle \xi,X_j(\ga) \rangle=0\quad \text{on }[0,1].
\end{equation}
Let us recall the definition of the end-point mapping with initial point
$x_0\in\R^n$. For any
$h\in L^2([0,1];\R^r)$, let $\ga^h$ be the solution of the problem
\[
 \dot\ga^h = \sum_{j=1}^r h_j X_j(\ga^h),\quad \ga^h(0)=x_0.
\]
The mapping $\mathcal E: L^2([0,1];\R^r)\to  \R^{n}$, $\mathcal E (h) =
\ga^h(1)$, is called 
the \emph{end-point mapping}  with initial point $x_0$.
It is well known that a  horizontal curve $\gamma$ starting from
$x_0$ with controls $h$ is an abnormal extremal if and only if
there exists $\lambda\in\R^{n}$, $\lambda\neq 0$,  
such that 
\begin{equation}\label{POX9}
    \langle d\mathcal E(h)v,\lambda\rangle =0
\end{equation}
for all $v\in L^2([0,1];\R^r)$. Here, $d\mathcal E(h)$ is the differential of
$\mathcal E$ at the point $h$.
Abnormal extremals are precisely the singular points of the end-point mapping.
Let $\mathrm{Im} \, d \mathcal{E}(h)\subset\Rn$ denote the image
of the differential $d\mathcal{E}(h):L^2([0,1];\R^r)\to\Rn$.
 
\begin{defi}\label{CORANK}
 The \emph{corank} of an abnormal extremal $\ga:[0,1]\to\Rn$ with controls $h$
is the
integer $n-\mathrm{dim}\big(
\mathrm{Im} \, d \mathcal{E}(h)\big) \geq 1$.
\end{defi}

\noindent
If $\ga$ has corank $m\geq 1$ then we have $m$ linearly independent functions
$\xi^1,\ldots,\xi^m\in\mathrm{Lip}([0,1];\Rn)$ each  solving the system
of adjoint equations \eqref{ADJ}.

The necessary condition \eqref{condnecessabnormali}
can be improved in the case of strictly abnormal minimizers.

\begin{teo}\label{teo:Goh}
Let $\ga:[0,1]\to\Rn$ be a strictly abnormal length minimizer. Then any dual
curve  $\xi\in\mathrm{Lip}([0,1],\Rn)$ satisfies
\begin{equation}\label{GOH}
\langle \xi, [X_i,X_j](\ga)\rangle = 0\quad\text{on }[0,1]
\end{equation}
for any $i,j=1,\dots, r$.
\end{teo}

\noindent Condition \eqref{GOH} is known as  {\em Goh condition}. Theorem
\ref{teo:Goh} can be deduced from second order open mapping theorems. We refer
to \cite[Chapter 20]{AS} for a systematic treatment of the
subject.

\begin{defi} \label{GOG-EX}
 We say that a horizontal curve $\ga:[0,1]\to\Rn$ is a \emph{Goh extremal} if
there exists a Lipschitz curve $\xi:[0,1]\to\Rn$ such that $\xi\neq0$, $\xi$
solves the adjoint equations \eqref{ADJ} and 
$\langle \xi, X_i(\ga)\rangle =\langle \xi, [X_i,X_j](\ga)\rangle =
0$  on $[0,1]$ for all $i,j=1,\ldots,r$.
\end{defi}

Goal of this paper is to integrate the system of   adjoint equations. We
shall actually integrate   an equivalent system of ordinary differential
equations. To this aim, let us complete the system of vector fields 
$X_1,\ldots,X_r$ to a frame
of $n$ linearly independent vector fields $X_1,\ldots,X_n$. 
This is always possible locally. Then there are smooth functions  $c_{ij}^k$
such that \begin{equation}\label{defalfaijk}
[X_i,X_j] = \sum_{k=1}^n c_{ij}^k X_k,\quad i,j=1,\dots,n.
\end{equation}
In Lie groups, if $X_1,\ldots,X_n$ form a basis of left invariant vector fields,
the functions $c^k_{ij}$ are constants   called {\em structure constants} of
the group.

Let $\theta_1,\ldots,\theta_n$  be the frame of $1$-forms dual to the
frame of vector fields $X_1,\dots,X_n$. In the standard basis $dx_1,\ldots,d
x_n$, we have for suitable functions $\theta_{ik}:\Rn\to\R$
\[
\theta_i=\sum_{k=1}^n \theta_{ik} dx_k.
\]
Then, for all $i,j=1,\ldots,n$, we have 
\begin{equation}\label{deltaij}
\delta_{ij} = \theta_i(X_j) 
%= \sum_{h,k=1}^n \theta_{ik} X_j^h\delta_{kh}
=\sum_{k=1}^n\theta_{ik} X_{jk}.
\end{equation}
The coefficients $X_{jk}$ are defined as in \eqref{COFFI}, for all
$j=1,\ldots,n$. 
Here and hereafter, $\delta_{ij}$ is the Kronecker symbol.

Given an extremal curve $\ga$ with dual curve $\xi$, let
$\la_1,\ldots,\la_n\in\mathrm{Lip}([0,1])$ be the functions defined via the
relation  along $\gamma$ 
\begin{equation}\label{omegalambda}
\xi_1 dx_1+\dots+\xi_n dx_n = \lambda_1
\theta_1(\gamma)+\dots+\lambda_n\theta_n(\gamma).
\end{equation}
We translate the adjoint equations \eqref{ADJ} for $\xi$ into a system of
differential equations for the coordinates $\la_1,\ldots,\la_n$ of $\xi$ in the
frame
$\theta_1,\ldots,\theta_n$.
We can express $\ga$ in the standard coordinates of $\Rn$ as 
$\gamma=(\gamma_1,\dots, \gamma_n)$. In the case of Lie groups, the following
theorem is proved in \cite{GK}.

\begin{teo}\label{equazionilambda}
Assume that the   vector fields $X_1,\ldots,X_r$ satisfy
\begin{equation}\label{coordorizzX_j}
X_{jh} = \d_{jh} \qquad\text{for all  }1\leq j,h\leq r.
\end{equation}
Then the functions $\xi_1,\ldots,\xi_n$ solve equations \eqref{ADJ} if and only
if
the functions
$\la_1,\ldots,\la_n$ satisfy the system of differential equations
\begin{equation}\label{8}
\dot\lambda_i = - \sum_{k=1}^n \sum_{j=1}^r  c_{ij}^k(\gamma) \dot \gamma_j 
\lambda_k\quad\text{a.e.~on }[0,1],
\end{equation}
for any $i=1,\dots,n$.
\end{teo}

\begin{proof} Let $h:[0,1]\to \R^r $ be the controls of
$\ga$:
\[
\dot\gamma= h_1 X_1(\gamma)+\dots+h_r X_r(\gamma)\quad\text{a.e.~on }[0,1].
\]
By \eqref{coordorizzX_j} we have $\dot\gamma_j=h_j$ a.e.~on $[0,1]$, 
$j=1,\dots,r$.
On the one hand, by differentiating the identity
\[
\xi_ k = \sum_{i=1}^ n \la_i \theta_{ik}(\ga),
\]
we obtain
\begin{align*}
\dot\xi_k =\,& \sum_{i=1}^n \dot\lambda_i\theta_{ik} + \lambda_i\langle
\nabla\theta_{ik}    (\gamma)   ,\dot\gamma\rangle\\
=\,&\sum_{i=1}^n \bigg(\dot\lambda_i\theta_{ik} + \lambda_i\Big\langle
\nabla\theta_{ik}(\gamma),\sum_{j=1}^r h_j X_j(\gamma)\Big\rangle\bigg)\\
=\,&\sum_{i=1}^n \bigg(\dot\lambda_i\theta_{ik} + \lambda_i\sum_{j=1}^r h_j\:
X_j\theta_{ik}(\gamma)\bigg)\,.\\
%\end{split}
%\]
\intertext{On the other hand, equation \eqref{ADJ} reads}
%\[
\dot \xi_k =\,& - \sum_{j=1}^r\sum_{i=1}^n h_j \frac{\de X_{ji}}{\de
x_k}(\gamma)
\sum_{h=1}^n \lambda_h \theta_{hi}(\gamma).
\end{align*}
Taking the difference of the previous two identities, we get
\[
0 = \sum_{i=1}^n \dot\lambda_i\theta_{ik} + \sum_{i=1}^n\sum_{j=1}^r\lambda_i
h_j\, X_j\theta_{ik}(\gamma) + \sum_{j=1}^r\sum_{i=1}^n\sum_{h=1}^n h_j
\frac{\de
X_{ji}}{\de x_k}(\gamma)  \lambda_h \theta_{hi}(\gamma),
\]
for any $k=1,\dots,n$. It follows that for any $p=1,\dots,n$ we have
\[
0 = \sum_{k=1}^n X_{pk}\left(\sum_{i=1}^n \dot\lambda_i\theta_{ik} +
\sum_{i=1}^n\sum_{j=1}^r\lambda_i h_j\, X_j\theta_{ik}(\gamma) +
\sum_{j=1}^r\sum_{i,h=1}^n h_j \frac{\de X_{ji}}{\de x_k}(\gamma)  \lambda_h
\theta_{hi}(\gamma)\right).
\]
Notice now that, by \eqref{deltaij},
\begin{equation}\label{deij2}
0 = \sum_{k=1}^n X_{pk} (X_j\theta_{ik}) + \sum_{k=1}^n (X_jX_{pk})
\theta_{ik}\:;
\end{equation}
moreover,
\begin{equation}\label{dede1}
\sum_{k=1}^n X_{pk} \frac{\de X_{ji}}{\de x_k}(\gamma) = (X_pX_{ji})(\gamma).
\end{equation}
Using \eqref{deltaij}, \eqref{deij2}, and \eqref{dede1}, we obtain
\[
\begin{split}
0 =\,& \dot\lambda_p - \sum_{i,k=1}^n\sum_{j=1}^r\lambda_i
h_j\,(X_jX_{pk})(\gamma)\, \theta_{ik}(\gamma) + \sum_{j=1}^r\sum_{i,h=1}^n
\lambda_hh_j (X_pX_{ji})(\gamma)\, \theta_{hi}(\gamma)\\
=\,& \dot\lambda_p - \sum_{i,k=1}^n\sum_{j=1}^r\lambda_i
h_j\,(X_jX_{pk})(\gamma)\, \theta_{ik}(\gamma) + \sum_{j=1}^r\sum_{k,i=1}^n
\lambda_i h_j (X_pX_{jk})(\gamma)\, \theta_{ik}(\gamma)\\
=\,& \dot\lambda_p + \sum_{i,k=1}^n\sum_{j=1}^r\lambda_i h_j
\theta_{ik}(\gamma)\big(X_pX_{jk} -  X_jX_{pk}\big)(\gamma).
\end{split}
\]
From \eqref{defalfaijk} we deduce that
\[
(X_pX_{jk} -  X_jX_{pk})(\gamma)  = \sum_{\ell=1}^n c_{pj}^\ell(\gamma) 
X_{\ell k}(\gamma),
\] 
and hence, by \eqref{deltaij}, we have
\[
0 = \dot\lambda_p + \sum_{i,k,\ell=1}^n\sum_{j=1}^r\lambda_i h_j
c_{pj}^\ell (\gamma)\theta_{ik}(\gamma) X_{\ell k}(\gamma)=\, \dot\lambda_p +
\sum_{i=1}^n\sum_{j=1}^r\lambda_i h_j c_{pj}^i(\gamma). 
\]
This is equation \eqref{8}. The same computations show that \eqref{ADJ} follows
from \eqref{8}.
\end{proof}

%%%%%%%%%%%

\section{Hall basis theorem for free nilpotent Lie algebras} \label{HALLE}
In this section, we develop the algebraic tools needed in Section \ref{agio} to
integrate the system of differential equations \eqref{8} in free
nilpotent  Lie groups. In groups, the structure functions $c_{ij}^k$
appearing in \eqref{defalfaijk} and \eqref{8} are in fact constants, as soon as
the vector fields $X_1,\ldots, X_n$ are left invariant. The technical
difficulty is to make transparent how the Jacobi identity enters the
integration process. The algebraic preliminaries for this integration are fixed
in Lemma \ref{Roberto}.

Let $\mathfrak g$ be a free  nilpotent  real Lie algebra of dimension $n$,
step $s\geq 2$ and rank $r\geq 2$. The algebra $\mathfrak g$ admits a 
stratification
\[
 \mathfrak g = \mathfrak g_{1} \oplus \cdots \oplus \mathfrak g_{s},
\]
where $\mathfrak g_{i+1} = [\mathfrak g_1,\mathfrak g_i]$ for   $i=1,\ldots,
s-1$ and
$\mathfrak g_i=\{0\}$ for $i>s$. The first layer $\mathfrak g_1$ has dimension
$r$.

We recall the algorithm for the construction of a basis for $\mathfrak g$ due to
M.~Hall Jr.~\cite{Hall}.
Let $X_1,\ldots,X_r$ be a basis for $\mathfrak 
g_1$. To each $X_i$ we assign the degree $d(i) = \deg(X_i)=1$ for
$i=1,\ldots,r$. By induction, we complete $X_1,\ldots,X_r$, to a basis
$X_1,\ldots,X_n$ for $\mathfrak g$, and we assign to each $X_j$ a degree $d(j)
= \deg(X_j) \in \{1,\ldots,s\}$. Let us assume that the elements
$X_1,\ldots.X_k$, with degree at most $d-1$, are already defined and that they
are ordered with the  property:
\[
d(i)\leq d-1,\ d(j)\leq d-1,\text{ and }d(i) < d(j)\ \Rightarrow\ i<j.
\]
By definition,  the commutator $[X_i,X_j]$ is an element of the
Hall basis of degree $d$ if:
\begin{subequations}%\label{2}
\begin{align}
&i>j,\\
&X_i \text{ and }X_j\text{  are elements of the basis of degree }\leq d-1,\\
 &d(i)+d(j)=d,\\
  &\text{if } X_i \text{ was constructed as }[X_h,X_k],
 %, \text{ i.e., } X_i
 %  \text{ was constructed as the bracket of the basis} \\
 % &\text{vectors } X_\alpha \text{ and } X_\beta,
  \text{ then }k\leq j.  
\end{align}
\end{subequations}
Hall proved that this algorithm produces a    basis $X_1,\dots,X_n$ for
$\mathfrak g$. The basis is ordered by subindices in such a way that $d(i) <
d(j)$ implies $i<j$.

Following the genesis of the basis, it is easy to see by induction on the degree
that, for any $\ell=1,\dots,n$ there exists a unique 
string  of indices ${\ell_0}, {\ell_1}, {\ell_2},   \dots,{\ell_h}\in
\{1,\ldots,n\}$, with $h\in \{0,\ldots, s-1\}$,
such that 
\begin{equation}\label{1}
X_\ell = [\cdots[[[X_{\ell_0},X_{\ell_1}],X_{\ell_2}],\dots],X_{\ell_h}]
\end{equation}
 and 
\begin{subequations}\label{2}
\begin{align}
& \ell_0>\ell_1,\label{2.1}\\
& d(\ell_0)=d(\ell_1)=1\label{2.2},\\
& \ell_1\leq \ell_2\leq\ldots\leq \ell_h,\label{2.3}\\
& \text{if }% k  \text{ is s.t.~}
X_k=[\cdots[[[X_{\ell_0},X_{\ell_1}],X_{\ell_2}],\dots],X_{\ell_{i-1}}], \text{
with } i=2,\ldots, h, \text{ then } \ell_i< k.\label{2.4}
\end{align}
\end{subequations}

We need some more notation. Let us first introduce the set of multi-indices
$\I= \N^n$.
For any  $\alpha\in \I$, we let $|\alpha| = \alpha_1+\ldots+\alpha_n$ and 
$u_\alpha=\max\{i:\alpha_i\neq 0\}$. Let us agree that $\alpha\leq\beta$
for $\alpha,\beta\in\I$ means 
$\alpha_i\leq\beta_i$ for all $i=1,\ldots, n$.

Fix a Hall basis $X_1,\ldots,X_n$.
For ${j_0}, {j_1}, {j_2},   \dots,{j_k}\in \{1,\ldots,n\}$, we let
\begin{equation}\label{notazione 1}
 [X_{j_0},X_{j_1},X_{j_2}, X_{j_3}, \dots,X_{j_k}]
 := [\cdots[[[X_{j_0},X_{j_1}],X_{j_2}],X_{j_3} ],\dots,X_{j_k}].
\end{equation}
 For 
$\alpha = (\alpha_1, \ldots,
\alpha_n)\in \I$,  we define the iterated commutator
\begin{equation}\label{commutatoreiterato}
[\cdot,X_\alpha]:= [  \cdot,\underbrace{X_1,\dots,X_1}_{\textrm{$\alpha_1$
times}},\underbrace{X_2,\dots,X_2}_{\textrm{$\alpha_2$
times}},   \dots,\underbrace{X_n,\dots,X_n}_{\textrm{$\alpha_n$
times}}].
\end{equation}
We agree that  $[\cdot,X_{(0,\ldots,0)}]={\rm Id}$.

For any $\ell\in\{1,\ldots,n\}$ there exist  unique $\ell_0\in\{2,\ldots,r\}$
and $h=h(\ell)\in \{1,\ldots,s-1\}$ such   that $X_\ell$ is given by the
representation \eqref{1} subject to \eqref{2.1}--\eqref{2.4}. For each
$\ell$ we define the
multi-index  $I(\ell)  \in\I$   as
\begin{equation}\label{multi-I}
 I(\ell)_j =    \# \{ s\geq 1: \ell_s=j\}.
\end{equation}
For example, if $X_\ell=[X_3,X_2]$, then $\ell_0=3$
and $I(\ell)=(0,1,0,\ldots,0)$. Also, if $\ell\in \{1,\ldots,r\}$ then
$\ell_0=\ell$ and 
$I(\ell)=(0,\ldots,0)$.

With this notation, we   have the following properties, for any
$\ell=1,\ldots,n$: 
\begin{equation}\label{IELLE}
X_\ell = [X_{\ell_0},X_{I(\ell)}],
\end{equation}
\begin{equation}
u_{I(\ell)}=\ell_{h(\ell)}.
\end{equation}

Moreover, given $\ell,k\in\{1,\ldots,n\}$,  we have 
\begin{equation}\label{328}
[X_\ell,X_k] \text{ is a base vector }
\iff
u_{I(\ell)}\leq k <\ell.
\end{equation}

\begin{defi}  
  \label{lop}
We say that $X_\ell=[X_{\ell_0},X_{\ell_1},X_{\ell_2},\dots,X_{\ell_h}]$ is a
direct discendant of $X_j$ if
$X_j=[X_{\ell_0},X_{\ell_1},X_{\ell_2},\dots,X_{\ell_k}]$, for some $k\in
\{0,\ldots, h\}$. In this case, we write $j\prec\ell$.
\end{defi}

\noindent The relation $\prec$ is a partial order on indices.
Notice that 
$j\prec\ell$ implies $I(j)\leq I(\ell)$,
but not viceversa.

The next combinatorial lemma plays a central role in the next section.
For any $\beta \in \I$ and $q=1,\ldots,r$, let us define the set of indices
\[
 \A_{\beta,q} = \big\{\ell \in\{1,\ldots,n\} : q\prec \ell, \,
I(\ell)\leq\beta\big\}.
\]
The statement $q\prec\ell$ is equivalent to $\ell_0=q$. We will also use the
notation
$\beta!:=\beta_1!\,\beta_2!\,\cdots\,\beta_n!$, $|\beta|
:=\beta_1+\cdots+\beta_n$, and
$\ex_k:=(\delta_{1k},\delta_{2k},\dots,\delta_{nk})=(0,\dots,1,\dots,0)$.

\begin{lemma}\label{Roberto} 
For all $\b\in\I$, $i=1,...,n$, and  $q=1,...,r$, we have 
\begin{equation}\label{RR}  
  [[X_i,X_q] ,X_\beta]
=
\sum_{ \ell\in \A_{ \beta,q } }
c_{\ell\beta}[X_i,X_{\beta-I(\ell)+\ex_\ell}],
\end{equation}
where we let  
\begin{equation}\label{cibeta}
c_{\ell\beta} = 
\dfrac{\beta!}{(\beta-I(\ell))! I(\ell)!} .
\end{equation}
\end{lemma}

\begin{proof}  
The case $\beta=0$ is straightforward. The proof is by induction on 
$|\beta|\geq 1$.

 \medskip

\textit{Base of induction.}  Assume $|\beta|=1$, i.e.,
$\beta=\ex_k$, for some 
$k\in\{1,2,\ldots,n\}$.
Hence,
$\beta!=1$. Let $\ell\in  \A_{ \beta,q } $.
Since  $I(\ell)\leq \beta$, then
either
$
I(\ell)=(0,\ldots,0)$ or $
I(\ell)=\ex_k$.
In both cases we have $c_{\ell\beta}=1$.
Since  $\ell_0=q$, then $X_\ell$ is either $X_q$, in which case
$\beta-I(\ell)+\ex_\ell=\ex_k+\ex_q$, or
$X_\ell=[X_q, X_k]$, in which case
$\beta-I(\ell)+\ex_\ell%=\ex_k-\ex_k+\ex_\ell
=\ex_\ell$.
Notice that the latter case is allowed only if $k<q $.

Therefore,  equation \eqref{RR}
reduces to a trivial identity when 
$k\geq q$, whereas when $k<q$ it reduces to 
$$
[ [X_i,X_q],X_k] =   [[X_i,X_k],X_q]  + [X_i,[X_q,X_k]],
$$ 
which is true by the Jacobi identity.

\medskip

\textit{Inductive step.} Using induction, we prove the claim for any 
$\widetilde \beta\in \I$. We write $\widetilde\beta=\beta+\ex_k$ for some
$\beta\in \I\setminus\{0\}$ and $k=u_{  \widetilde\beta }  \in\{1,2,\ldots,n\}$.
Hence we have
 $u_\beta\leq k$ and $|\beta|<|\widetilde\beta|$. By induction, formula
\eqref{RR}
holds for $\beta$ and thus:
\[   [[X_i,X_q] ,X_{\widetilde\beta}] 
%=
% \frac{1}{ (\beta_k+1)\, \b! } [[[X_i,X_q] ,X_{\beta}],X_k]
 = 
\sum_{\ell\in\A_{\beta,q} } c_{\ell\beta}
[[X_i,X_{\beta-I(\ell)+\ex_\ell}],X_k].
\]

We split the sum into two sums according to whether $\ell\leq k$ or $\ell >k$.
In the case when $\ell\leq k$,  the coordinates of the multi-index
$\beta-I(\ell)+e_\ell$
 from the $(k+1)$-th onward are null and hence, 
\[ 
 [[X_i,X_{\beta-I(\ell)+\ex_\ell}],X_k]= 
[X_i,X_{\beta-I(\ell)+\ex_\ell+\ex_k}].
\]
In the case  $\ell >k$, also because  $u_{I(\ell)}\leq k$, the commutator
$[X_\ell, X_k]$ is an element of the Hall basis, see \eqref{328}. We denote
this element by  $X_{\ell.k}= [X_\ell, X_k]$. Then by the  Jacobi identity we
obtain
\[\begin{split}
[[X_i,X_{\beta-I(\ell)+\ex_\ell}],X_k]&=[[[X_i,X_{\beta-I(\ell)}],
X_\ell],X_k]\\
&=
[[[X_i,X_{\beta-I(\ell)}], X_k],X_\ell] +[[X_i,X_{\beta-I(\ell)}], [X_\ell,X_k]]
\\
&=
[X_i,X_{\beta-I(\ell)+\ex_k+\ex_\ell}] +[X_i,X_{\beta-I(\ell)+\ex_{\ell.k}}]. 
\end{split}
\]
To conclude the proof, we need to show that the sum
\begin{equation}\label{RHS}
\sum_{\ell\in\A_{\beta,q} } c_{\ell\beta}
[X_i,X_{\beta-I(\ell)+\ex_\ell+\ex_k}]
+
\sum_{ \ell\in\A_{\beta,q}, \ell>k } c_{\ell\beta} 
[X_i,X_{\beta-I(\ell)+\ex_{\ell.k}}] 
 \end{equation}
is equal to $\sum_{  \ell\in \A_{\beta+\ex_k,q}} c_{\ell,\beta+\ex_k}  
[X_i,X_{\beta+\ex_k-I(\ell)+\ex_\ell}]$, i.e., to
\begin{equation}\label{LHS}
\sum_{\substack{  \ell\in \A_{\beta+\ex_k,q}\\I(\ell)_k\leq\beta_k}}
c_{\ell,\beta+\ex_k}  [X_i,X_{\beta+\ex_k-I(\ell)+\ex_\ell}] 
+ 
\sum_{\substack{  \ell\in \A_{\beta+\ex_k,q}\\I(\ell)_k=\beta_k+1}}
c_{\ell,\beta+\ex_k}  [X_i,X_{\beta+\ex_k-I(\ell)+\ex_\ell}]\,.
\end{equation}
For fixed $\beta$, we introduce the notation
\[
\Phi(\ell)= c_{\ell\beta} [X_i,X_{\beta+\ex_k-I(\ell)+\ex_\ell}],
\] 
that is well defined if $I(\ell)_k\leq\beta_k$. 

Let us rewrite \eqref{RHS}. In the second summation in \eqref{RHS}, we perform
the change of indices
$\widetilde\ell=\ell.k$. Then we have $I(\widetilde \ell) = I(\ell)+\ex_k$ and
the summation becomes
\[
\begin{split}
&\sum_{\substack{\widetilde \ell\in \A_{\beta+\ex_k,q}\\ I(\widetilde\ell)_k\geq
1 }}
 \dfrac{\beta!}{(\beta+\ex_k-I(\widetilde\ell))! (I(\widetilde\ell)-\ex_k)!}
[X_i,X_{\beta+\ex_k-I(\widetilde\ell)+\ex_{\widetilde\ell}}] \\
= &
\sum_{\substack{\widetilde \ell\in \A_{\beta+\ex_k,q}\\ 1\leq
I(\widetilde\ell)_k\leq \beta_k }}
 \dfrac{I(\widetilde\ell)_k}{\beta_k+1-I(\widetilde\ell)_k } 
\Phi(\widetilde\ell)
+
\sum_{\substack{\widetilde \ell\in \A_{\beta+\ex_k,q}\\  I(\widetilde\ell)_k =
\beta_k +1}}
\frac{\beta!}{(\beta+\ex_k-I(\widetilde\ell))!(I(\widetilde\ell)-\ex_k)!}[X_i,X_
{\beta+\ex_k-I(\widetilde\ell)+\ex_{\widetilde\ell}}]
. 
\end{split}
\]
Therefore, the sum in \eqref{RHS} is 
\begin{equation}\label{stel}
\begin{split}
\eqref{RHS} =\ & \sum _{\ell\in\A_{\beta,q}}
   \Phi(\ell)
 +
 \sum_{\substack{\ell\in \A_{\beta+\ex_k,q}\\ 1\leq I(\ell)_k\leq \beta_k }}
 \dfrac{I(\ell)_k}{\beta_k+1-I( \ell)_k }  \Phi( \ell)\\
& +
 \sum_{\substack{  \ell\in \A_{\beta+\ex_k,q}\\I(\ell)_k=\beta_k+1}} 
\frac{\beta!}{(\beta+\ex_k-I(\ell))!(I(\ell)-\ex_k)!} 
[X_i,X_{\beta+\ex_k-I(\ell)+\ex_\ell}]. 
\end{split}
\end{equation}
We split the first summation in \eqref{stel} according to whether
$I(\ell)_k=0$ or not.
Notice that $\ell\in \A_{\beta,q}$ is equivalent to $\ell \in
\A_{\beta+\ex_k,q}$ and $I(\ell)_k \leq \beta_k$. Then we have
\[
\begin{split}
 \sum _{\ell\in\A_{\beta,q}}   \Phi(\ell)  = &
 \sum _{\substack{ \ell\in\A_{\beta+\ex_k,q}\\ I(\ell)_k = 0}} \Phi(\ell)
 +    \sum _{\substack{ \ell\in\A_{\beta+\ex_k,q}\\1\leq I(\ell)_k\leq\beta_k}}
   \Phi(\ell)\\
= & \sum _{\substack{ \ell\in\A_{\beta+\ex_k,q}\\ I(\ell)_k = 0}} \dfrac{\beta_k
+ 1}{\beta_k+1-I( \ell)_k }\Phi(\ell)
 +  \sum _{\substack{ \ell\in\A_{\beta+\ex_k,q}\\1\leq I(\ell)_k\leq\beta_k}}
   \Phi(\ell)
\end{split}
\]
and, by \eqref{stel}, we obtain
\[
\begin{split}
\eqref{RHS} =\ & 
 \sum _{\substack{ \ell\in\A_{\beta+\ex_k,q}\\ I(\ell)_k = 0}} \dfrac{\beta_k +
1}{\beta_k+1-I( \ell)_k }\Phi(\ell)
 +    \sum _{\substack{ \ell\in\A_{\beta+\ex_k,q}\\1\leq I(\ell)_k\leq\beta_k}}
\Phi(\ell)
 +
 \sum_{\substack{\ell\in \A_{\beta+\ex_k,q}\\ 1\leq I(\ell)_k\leq \beta_k }}
 \dfrac{I(\ell)_k}{\beta_k+1-I( \ell)_k }  \Phi( \ell)\\
 &+\sum_{\substack{  \ell\in \A_{\beta+\ex_k,q}\\I(\ell)_k=\beta_k+1}} 
\frac{\beta!}{(\beta+\ex_k-I(\ell))!(I(\ell)-\ex_k)!} 
[X_i,X_{\beta+\ex_k-I(\ell)+\ex_\ell}]\,.
\end{split}
\]
Hence, the sum in \eqref{RHS} is
\[
\begin{split}
 \eqref{RHS} =\ &
 \sum_{\substack{\ell\in \A_{\beta+\ex_k,q}\\ 0\leq I(\ell)_k\leq \beta_k }}
 \dfrac{\beta_k + 1}{\beta_k+1-I( \ell)_k }  \Phi( \ell)\\
& +
 \sum_{\substack{  \ell\in \A_{\beta+\ex_k,q}\\I(\ell)_k=\beta_k+1}} 
\frac{\beta!}{(\beta+\ex_k-I(\ell))!(I(\ell)-\ex_k)!} 
[X_i,X_{\beta+\ex_k-I(\ell)+\ex_\ell}]
\end{split}
\]
and it can be easily checked that the right hand side of the previous formula
equals \eqref{LHS}. This concludes the proof.
\end{proof}

\medskip

Let  $\mathfrak f_{s,r}$ denote   the free nilpotent  Lie algebra over $\R$
of step $s$ and rank $r$. Let $n =\mathrm{dim}(\mathfrak f_{r,s})$ be the
dimension of the algebra.
Let us fix a Hall basis for $\mathfrak f _{s,r}$. Then
for any
$\ell\in\{1,\ldots,n\}$ we have a  multi-index  $I(\ell)\in\I$,
see \eqref{multi-I}, and moreover there is a partial order $\prec$ on indices,
see Definition \ref{lop}.
Finally, recall the notation $x^\alpha = x_1^{\alpha_1}\cdot\ldots\cdot
x_n^{\alpha_n}$
for $x\in\Rn$ and $\alpha\in\I$.

The following theorem is
proved by M. Grayson and R. Grossman, see \cite[Theorem 2.1]{GrGr}.

\begin{teo}\label{cerchio}
 The vector fields $X_1,\ldots,X_r$ in $\Rn$
\begin{equation}
 \label{quadrato}
  X_i(x) = \sum_{\ell : i\prec\ell} \frac{(-1)^{|I(\ell)|} } {I(\ell)!}
x^{I(\ell)}
\frac{\partial }{\partial x_{\ell}},\quad x\in\Rn,
\end{equation}
with $i=1,\ldots,r$, generate a Lie algebra isomorphic to $\mathfrak f_{s,r}$.
\end{teo}

\noindent 
The index $\ell=i$ is included in the sum in \eqref{quadrato} and gives the
summand
$\partial/\partial x_i$. In \cite{GrGr}, the authors use a slightly different
notation. Moreover, there is no $\ell\neq 1$ such that $1\prec
\ell$ and thus $X_1= \partial /\partial x_1$. Finally, notice that
$X_1,\ldots,X_r$ satisfy  assumption \eqref{coordorizzX_j}.

By Hall's construction, the generators $X_1,\ldots,X_r$ can be completed to
a basis $X_1,\ldots,X_n$ of the Lie algebra. We call 
 $X_1,\ldots,X_n$ a {\em Hall-Grayson-Grossman basis} of vector fields in $\Rn$.

In \cite{GrGr}, the authors also describe the elements $X_j$ with $j>r$.
For $\alpha\in\I$ with $\alpha\neq 0$, define the {\em minimal order} of the
monomial
$x^\alpha$ as
\[
m(x^\alpha)=\min\big\{ j\in\{1,\ldots,n\} : \alpha_j>0\big\}.
\]
For a polynomial $P(x) = \sum_{h=1}^N c_ h x^{\alpha_h}$ with $c_h\neq0$ and
$\alpha_h\in\I$, $\alpha_h\neq0$, we define the minimal order $m(P) =
\max_{h=1,\ldots,N} m(x^{\alpha_h})$.

For $i,\ell\in\{1,\ldots,n\}$ with $i\prec\ell$, let us define the monomial
\begin{equation}
 \label{star}
  P_{i\ell}(x) = \frac{(-1)^{|I(\ell)|-|I(i)|}}{(I(\ell)-I(i))!} 
x^{I(\ell)-I(i)},\quad x\in\Rn.
\end{equation}
Notice that \eqref{quadrato} can be rewritten as
\[
X_i(x)=\sum_{\ell:i\prec \ell} P_{i\ell}(x)\frac{\partial}{\partial
x_\ell},\qquad i=1,\dots,r.
\]

\begin{lemma}{\rm(}\cite[Lemma 2.3]{GrGr}{\rm)}
If $X_i$ is an element of the Hall basis constructed as $X_i = [X_j,X_k]$, then 
\begin{equation}
 \label{tri}
  X_i(x) = \sum_{\ell : i\prec\ell} P_{i\ell}(x)\frac{\partial}{\partial
x_\ell} 
 + \sum_{\ell=1}^n Q_{i\ell}(x) \frac{\partial}{\partial x_\ell},\quad x\in\Rn,
\end{equation}
where $Q_{i\ell}$ are polynomials with no constant terms satisfying
$m(Q_{i\ell})<k$ and $P_{i\ell}$ are monomials of the form \eqref{star}
satisfying $k\leq m (P_{i\ell})<i$ for all $\ell\neq i$.
\end{lemma}

The exponential mapping of the second type  related to an ordered system of
vector
fields $X_1,\ldots,X_n$ is the mapping $\Psi:\Rn\to\Rn$, when globally defined,
\begin{equation}
 \label{PSI}
   \Psi(x) = \ex^{x_1 X_1}\circ\ldots\circ \ex^{x_n X_n} (0),\quad x\in\Rn,
\end{equation}
where $\ex^{t X}$ denotes the flow of the vector field $X$ with parameter
$t\in\R$.
When $G$ is a Lie group with group law $\cdot$ and $X_1,\ldots,X_n\in
\mathfrak g=\mathrm{Lie}(G)$, we can equivalently define 
the mapping $\Psi: \Rn\to G$ as 
\begin{equation}
 \label{PSI2}
 \Psi(x) = \exp(x_n X_n)\cdot \ldots\cdot \exp(x_1 X_1) ,\quad x\in \R^n,
\end{equation}
where $\exp:\mathfrak g \to G$ is the exponential mapping.
 
We prove that a Hall-Grayson-Grossman basis of vector fields in $\Rn$ induces
exponential coordinates of the second type.

\begin{prop}\label{EXPO}
 Let $X_1,\ldots,X_n$ be a Hall-Grayson-Grossman basis of vector fields in
$\Rn$. Then we have $x=\Psi(x)$ for all $x\in\Rn$.
\end{prop}

\begin{proof}
 Let $\ex_1,\ldots,\ex_n$ be the standard basis of $\Rn$. By the structure
\eqref{tri} of $X_n$, we deduce that $\ex^{x_n X_n}(0) = x_n\ex_n$. Assume by
induction that for some $1\leq i<n$ we have
\begin{equation}\label{INDUX}
  \Psi(x) = \ex^{x_1 X_1}\circ\ldots\circ \ex^{x_i X_i} \Big( \sum_{h=i+1}^n
x_h \ex_{h}\Big). 
\end{equation}
Assume first that $i>r$. Then we have $X_i = [X_j,X_k]$ for some $k<j<i$, and
$X_i$ is of the form \eqref{tri}. Since $m(Q_{i \ell})<k<i$ for all
$\ell=1,\ldots,n$, the polynomial $Q_{i \ell}$ vanishes along the flow of $X_i$
starting from $\sum_{h=i+1}^n x_h \ex_h$. Because $m(P_{i \ell})<i$ for
$\ell\neq i$ with $i\prec \ell$, the monomial $P_{i \ell}$ also vanishes along
the same flow. When $\ell=i$, we have $P_{i \ell}=1$. It follows that
\begin{equation}
 \label{INTER}
 \ex^{x_i X_i} \Big( \sum_{h=i+1}^n
x_h \ex_{h}\Big)=\sum_{h=i}^n
x_h \ex_{h}.
\end{equation}
This proves the inductive step when $i>r$.

Assume now $i\in\{ 1,\ldots,r\}$. We have $X_1=\partial/\partial x_1$ and for
$i=2,\ldots,r$ the condition $i\prec\ell$, $\ell\neq i$, implies that
$m(P_{i \ell})<i$. The same argument as above proves \eqref{INTER}
also when $i=1,\ldots,r$.
\end{proof}

%%%%%
\section{Integration of the adjoint equations in free nilpotent groups}
\label{agio}

Let $G$ be a free nilpotent Lie group of dimension $n$ and rank $r$. The Lie
algebra $\mathfrak g$ of $G$
is isomorphic to a Lie algebra of  vector fields in $\Rn$  that are left
invariant with respect to some product structure.
Let $X_1,...,X_n$ be a Hall-Grayson-Grossman basis of this Lie algebra.
Recall that the generators $X_1,\ldots,X_r$ have the form \eqref{quadrato}.
In this section, we integrate the system of differential equations \eqref{8} in
$\Rn$ with the structure constants $c_{ij}^k \in\R$ determined by the basis
$X_1,...,X_n$:
\begin{equation}\label{StrCon}
 [X_i,X_j] = \sum_{k=1}^n c_{ij}^k X_k,\quad i,j=1,\ldots,n.
\end{equation}
Let us also introduce the \emph{generalized structure constants}
 $c_{i\al}^k\in\R$ for any multi-index $\a\in\mathcal I = \N^n$ and 
$i,k\in\{1,\dots,n\}$. These constants are defined via the relation
\begin{equation}\label{GSC}
  [X_i,X_\a]= \sum_{k=1}^n c^k_{i\a} X_k.
\end{equation}
Recall the definition \eqref{commutatoreiterato} for the  iterated commutator 
$[X_i,X_\a]$. 

For any $i=1,\ldots,n$ and $\al\in\I$, define the linear mapping
$\phi_{i\al}:\Rn\to\R$
\begin{equation}\label{LinM}
 \phi_{i\al}(v) = \frac{(-1)^{|\a|}}{\a!}  \sum_{k=1}^n
 c^k_{i\a}v_k,\quad v=(v_1,\ldots,v_n)\in\Rn.
\end{equation}
Notice that $\phi_{i0}(v)=v_i$. The following polynomials are central objects in
the integration formulae for
the adjoint equations.

\begin{defi}
For each $i\in\{1,\dots,n\}$ and $v\in\R^n$, we call the polynomial
$P_i^v:\Rn\to\R$
\begin{equation}\label{gestalt}
P_i^v(x)= \sum_{\a\in \I} \phi_{i\alpha}(v) 
 x^\a,\quad x\in\Rn,
\end{equation}
\noindent
\emph{extremal polynomial}
of the free nilpotent group $G$ with respect to the basis $X_1,\ldots,X_n$ of
$\mathfrak g =\mathrm{Lie}(G)$.
\end{defi}

\begin{remark} Notice that the generalized structure constants $c_{i\alpha}^k$
in
\eqref{GSC} satisfy 
 $c_{i\alpha}^k =0 $
if $d(i)+ |\alpha|>s$, where
$s$ is the step of the Lie algebra. Then the  polynomial $P^v_i(x) $ has
homogeneous degree
at most $s-d(i)$. Recall that, by definition, the homogeneous degree of a
monomial $x^\alpha,\alpha\in\I$, is $d(\alpha):=\sum_{j=1}^n \alpha_j\,d(j)$ and
the homogeneous degree of a polynomial $\sum_{i=1}^N c_ix^{\alpha_i}$, with
$c_i\neq0$ and $\alpha_i\in\I$, is $\max_{i=1,\dots,N} d(\alpha_i)$.
\end{remark}
  
\begin{prop}\label{lemma-0}
Extremal polynomials have the following properties. 

{\upshape (i)} If $\ell,i=1,\ldots,n $  are  such
that $i\prec \ell$ and $v_\ell\neq 0$, then $P_i^v\neq 0$. 

{\upshape (ii)} If $v\in\R^n$ is such that $P_i^{v}=0$ for all
$i=1,\dots,r$, then $v=0$.

%{\upshape (ii)} Let $r_2:=\dim \mathfrak g_2$. If $v\in\R^n$ is such that
$P_i^{v}=0$ for all
%$i=r+1,\dots,r_2$, then $v_{r+1}=v_{r+2}=\cdots=v_n=0$.

{\upshape (iii)} For all  $i=1,\ldots,n $ and $v\in\Rn$ we have
$P_i^v(0)=v_i$. 
\end{prop}

\begin{proof} 
Let us prove (i). As  $i\prec\ell$,  we have $X_\ell=[X_i,X_{\al}]$ for 
$\al=I(\ell)\in\I$. It follows that 
$c_{i \al}^k=\delta_{\ell k}$, and thus, as $v_\ell\neq0$,
\[
P_i^v(x)=\frac{(-1)^{|\alpha|} }  {\alpha!}  v_\ell
x^{\al} + \sum_{ \beta \in\I, \, \beta \neq\al} \phi_{i\beta}(v)  
x^\beta\neq 0.
\]
Statement (ii) is an easy consequence of (i), while the proof of (iii) is
elementary and relies upon the identity
$c^k_{i0}=\delta_{ik}$.
\end{proof}

\begin{remark}\label{lemma-01}
We point out the following implication: if $r_2:=\dim \mathfrak g_2$ and
$v\in\R^n$ is such that $P_i^{v}=0$ for all
$i=r+1,\dots,r+r_2$, then $v_{r+1}=v_{r+2}=\cdots=v_n=0$. 

Indeed, arguing by contradiction we assume that there exists $j$ with $d(j)\geq
2$ such that $v_j\neq 0$. By Proposition \ref{lemma-0} (iii) we have
\[
v_{r+1}=v_{r+2}=\cdots=v_{r+r_2}=0,
\]
hence $d(j)\geq 3$ and we can write
\[
X_j=[\cdots[X_{j_0},X_{j_1}],\dots],X_{j_h}]
\]
for a suitable $h\geq 2$. The element $[X_{j_0},X_{j_1}]$ belongs to the Hall
basis and, in particular, $[X_{j_0},X_{j_1}]=X_k$ for a suitable $k$ of degree
2. Since $k\prec j$, Proposition \ref{lemma-0} (i) gives $P_k^v\neq 0$, a
contradiction.
\end{remark}

The following lemma shows that the linear subspace generated by extremal
polynomials is closed under derivatives along left invariant vector fields, in a
way related to the structure of $\mathfrak g$.

\begin{prop}\label{propderivpolinomi}
For any $v\in\R^n$ and $i,j\in\{1,\dots,n\}$ there holds 
\begin{equation}\label{derivatepolinomi}
X_i P_j^v = \sum_{k=1}^n c_{ij}^k P_k^v.
\end{equation}
In particular, $X_iP_i^v=0$ for any $i=1,\dots,n$ and $v\in\R^n$.
\end{prop}

\begin{proof}
Let $v\in\R^n$ and $j\in\{1,\dots,n\}$ be fixed. We argue  by
induction on the degree of $i$. If $d(i)=1$, i.e., $i=1,\ldots,r$,  we
have by \eqref{quadrato}
\[
\begin{split}
X_i P_j^v = & \left(\sum_{i\prec \ell} \frac{(-1)^{|I(\ell)|}}{I(\ell)!}
x^{I(\ell)} \frac{\partial}{\pa x_\ell}\right)
\sum_{\al\in\I}\phi_{j\alpha}(v) x^\al\\
= & \sum_{k=1}^nv_k \sum_{\al\in\I} \sum_{i\prec \ell} c_{j\al}^k  
\al_\ell \frac{(-1)^{|\alpha|+ |I(\ell)|}}{\alpha!I(\ell)!}
x^{\al-\erm_\ell+I(\ell)}.
\end{split}
\]
We now use the following polynomial identity
\begin{equation}\label{eqcruciale2}
\sum_{\a\in \I} 
\sum_{i\prec \ell}
c^k_{j\a}   
\a_\ell 
\frac{(-1)^{|\alpha|+ |I(\ell)|}}{\alpha! I(\ell)!}
x^{\a-\ex_\ell+I(\ell)} 
=
-  \sum_{h=1}^n \sum _{\beta\in\I} c_{ji}^h  c^k_{h\beta}
\frac{(-1)^{|\beta|}}{\beta!} x^\beta.
\end{equation}
This identity will be proven later, in a crucial
step of the proof of Theorem
\ref{teointduale}, see equation \eqref{eqcruciale}.
From  \eqref{eqcruciale2} we obtain
\[
\begin{split}
X_i P_j^v = & - \sum_{k=1}^n v_k \sum_{h=1}^n \sum_{\be\in\I} c_{ji}^h
c_{h\be}^k \frac{(-1)^{|\beta|}}{\beta!} x^\be 
=  \sum_{h=1}^n c_{ ij}^h P_h^v.
\end{split}
\]
This ends the proof of the induction base.

If $d(i)\geq 2$, we have $X_i=[X_p,X_u]$ for some $p,u$ with $d(p),d(u)<d(q)$.
By inductive assumption, we have
\[
\begin{split}
X_iP_j^v = & X_p X_u P_j^v - X_u X_p P_j^v\\
= & X_p \Big(\sum_{k=1}^nc_{uj}^k P_k^v\Big) - X_u\Big(\sum_{k=1}^nc_{pj}^k
P_k^v\Big)\\
= & \sum_{h,k=1}^n (c_{uj}^kc_{pk}^hP_h^v - c_{pj}^kc_{uk}^hP_h^v)\\
= & -\sum_{h,k=1}^n (c_{ju}^kc_{pk}^h  + c_{pj}^kc_{uk}^h)P_h^v.
\end{split}
\]
The Jacobi identity 
$[[X_j,X_u],X_p]+[[X_u,X_p],X_j]+[[X_p,X_j],X_u]=0$ yields
\[
\sum_{k=1}^n c_{ju}^k c_{kp}^h+c_{up}^k c_{kj}^h+ c_{pj}^k c_{ku}^h=0.
\]
Using this identity and
$c_{up}^kc_{jk}^h=c_{pu}^k c_{kj}^h=\delta_{ik}c_{kj}^h$, we obtain

\[
X_iP_j^v =\sum_{h,k=1}^n  c_{up}^k c_{jk}^h P^v_h=\sum_{h=1}^n 
c_{ij}^h P^v_h.
\]
 This completes the proof.
\end{proof}

We identify a free nilpotent Lie group $G$ with $\Rn$ via exponential
coordinates of the second type related to a Hall-Grayson-Grossman basis
$X_1,\ldots,X_n$, as explained in Proposition \ref{EXPO}.

\begin{teo}\label{teointduale}
Let $G=\Rn$ be a free nilpotent Lie group, let $\ga:[0,1]\to G$
be a horizontal curve such that
$\ga(0)=0$, and let $\la:[0,1]\to\R^n$ be a
Lipschitz curve. The following statements are
equivalent:

\begin{itemize}
 \item[A)] The curve $\la$ solves the system of equations \eqref{8}.

 \item[B)] There exists $v\in\Rn$  such that, for all $i=1,\ldots,n$,
we have
\begin{equation}\label{eqlaP}
\la_i(t) = P_i^{v}(\ga(t)),\quad t\in[0,1],
\end{equation}
and in fact $v=\la(0)$.
\end{itemize}
\end{teo}

\begin{proof} We show that the curve $\la$ defined by 
\eqref{eqlaP}  solves the system \eqref{8} with initial condition 
$\la(0)=v$. By the uniqueness of the solution, this will
prove the equivalence
of the statements A) and B). Notice that the curve $\la$ defined via
\eqref{eqlaP} satisfies $\la(0)=v$ by Proposition \ref{lemma-0} part (iii).

Recall that $\gamma$ is Lipschitz-continuous and is therefore differentiable
almost everywhere. 
We preliminarily   compute the derivative of $t\mapsto \gamma(t) ^\al$ for
any multi-index $\al\in\I$, at any differentiability point. We have
\begin{equation}\label{DERO1}
\frac{d}{dt} \ga^\a = \sum_{\ell=1}^n \a_\ell \ga_\ell^{\a_\ell-1} \dot \ga_\ell
\prod_{p\neq \ell} \ga_p^{\a_p} = \sum_{\ell=1}^n \a_\ell 
\ga^{\a-\ex_\ell}\dot \ga_\ell.
\end{equation}

Since $\ga$ is horizontal we have  
\begin{equation}\label{popis}
\dot\ga=\sum_{q=1}^r \dot\ga_q X_q(\ga)\qquad\text{a.e. on }[0,1],
\end{equation}
where $X_1,\dots,X_n$ is the fixed Hall basis and $r\geq 2$
is the rank of the group.

For any $\ell=1,\dots,n$ we can compute the derivative $\dot\ga_\ell$
starting from \eqref{popis} and from the formula \eqref{quadrato} for
the generators $X_1,\ldots,X_r$ of the Lie algebra. We have to consider the 
index $\ell_0=1,\ldots,r$ and the multi-index $I(\ell)\in \I$ such that
$X_\ell=[X_{\ell_0}, X_{I(\ell)}]$, see \eqref{IELLE}, and then look at the
$\ell$-th coordinate of the vector field $X_{\ell_0}$. We   then find 
\begin{equation}\label{focit}
\dot\ga_\ell =
%\dot\ga_{\ell_0}\,P_{\ell_0,\ell}(\ga)=
\dot\ga_{\ell_0}\,\frac{(-1)^{|I(\ell)|}
} {I(\ell)!} \ga^{I(\ell)} \qquad\text{a.e. on }[0,1].
\end{equation}
By \eqref{eqlaP},  \eqref{DERO1},  \eqref{focit}
 and the definition of $P^v_i$, we obtain
\[
\begin{split}
\dot \la_i & = \sum_{\a\in \I}  \sum_{j=1}^n\sum_{\ell=1}^n  c^j_{i\a} v_j 
\a_\ell   \dot \ga_{\ell_0}\,  
 \frac{(-1)^{|\al|+|I(\ell)|}}{\al! I(\ell)!}\,\ga^{I(\ell)+\a-\ex_\ell},
\end{split}
\]
which is equivalent to  
\begin{equation}\label{8mubis}
\dot \la_i  = \sum_{q=1}^r \dot \ga_q \sum_{j=1}^n v_j \sum_{\a\in \I} 
\sum_{\ell:\ell_0=q} c^j_{i\a}    \a_\ell
\frac{(-1)^{|\al|+|I(\ell)|}}{\al! I(\ell)!}  \ga^{\a-\ex_\ell+I(\ell)}.
\end{equation}
On the other hand, the right hand side of \eqref{8} is 
\begin{equation}\label{8muter}
- \sum_{q=1}^r \dot \ga_q \sum_{k=1}^n  c_{iq}^k \la_k
 = - \sum_{q=1}^r \dot \ga_q \sum_{j=1} ^n v_j\sum_{k=1}^n \sum _{\beta
\in\I} c_{iq}^k  c^j_{k\beta} \frac{(-1)^{|\beta|}}{\beta!} \ga^\beta.
\end{equation}
Comparing the lines   \eqref{8mubis} and \eqref{8muter}, we see that it is
sufficient  to show that for any fixed $i,j=1,\dots,n$ and
$q=1,\dots,r$ the following identity of polynomials is
verified
\begin{equation}\label{eqcruciale}
\sum_{\a\in \I} 
\sum_{\ell:\ell_0=q}
c^j_{i\a}  
\a_\ell 
\frac{(-1)^{|\al|+|I(\ell)|}}{\al!I(\ell)!}
x^{\a-\ex_\ell+I(\ell)} 
=
-  \sum_{k=1}^n \sum _{\beta\in\I} c_{iq}^k 
c^j_{k\beta}\frac{(-1)^{|\beta|}}{\beta!}  x^\beta.
\end{equation}
On the other hand, this identity is equivalent to the following combinatorial
identity, for any fixed $\beta\in\I$: 
\[
\sum_{\substack{\ell:\ell_0=q\\ \al:\a-\ex_\ell+I(\ell)=\beta}}
c^j_{i\a}  
\a_\ell 
\frac{(-1)^{|\al|+|I(\ell)|}}{\al!I(\ell)!}
=
-  \sum_{k=1}^n  c_{iq}^k  c^j_{k\beta}\frac{(-1)^{|\beta|}}{\beta!} .
\]
Since $X_1,\ldots,X_n$ is a basis, this is in turn equivalent to
\[
\begin{split}
\sum_{j=1}^n \Big(\sum_{\substack{\ell:\ell_0=q\\
\al:\a-\ex_\ell+I(\ell)=\beta}}
c^j_{i\a}   
\a_\ell 
\frac{(-1)^{|\al|+|I(\ell)|}}{\al! I(\ell)!} \Big )X_j
=  & -  \sum_{k=1}^n  c_{iq}^k
\frac{(-1)^{|\beta|}}{\beta!} 
  \sum_{j=1}^n c^j_{k\beta} X_j\\
= & -  \sum_{k=1}^n \frac{(-1)^{|\beta|}}{\beta!}  c_{iq}^k   [X_k,X_\beta]\\
= & -    \frac{(-1)^{|\beta|}}{\beta!}  [[X_i,X_q] ,X_\beta].
\end{split}
\]
We may rearrange the last identity in the following way:
\begin{equation}\label{19marzo}
\sum_{\substack{\ell:\ell_0=q\\ \al:\a-\ex_\ell+I(\ell)=\beta}}
 \a_\ell \frac{(-1)^{|\al|+ |I(\ell)|}}{\a! I(\ell)!} [X_i,X_\a]
= - \frac{(-1)^{|\be|}}{\be!} [[X_i,X_q] ,X_\beta].
\end{equation}

\noindent 
Notice that the condition  $\a-\ex_\ell+I(\ell)=\beta$ implies 
$I(\ell)\leq\beta$, because $I(\ell)_\ell=0$. In particular, the left hand side
of \eqref{19marzo} is
\[
\sum_{\substack{\ell:\ell_0=q\\ \al:\a-\ex_\ell+I(\ell)=\beta}}
 \a_\ell \frac{(-1)^{|\al|+ |I(\ell)|}}{\a! I(\ell)!} [X_i,X_\a]
=
 -\sum_{\substack{\ell:\ell_0=q\\ I(\ell)\leq\beta}}
\frac{(-1)^{|\be|}}{(\be-I(\ell))! I(\ell)!} [X_i,X_{\beta-I(\ell)+\ex_\ell}].
\] 
We conclude that \eqref{19marzo} is equivalent to the identity 
\[
\sum_{\substack{\ell:\ell_0=q\\ I(\ell)\leq\beta}}
\frac{1}{(\be-I(\ell))! I(\ell)!} [X_i,X_{\beta-I(\ell)+\ex_\ell}] 
=
\frac{1}{\be!} [[X_i,X_q] ,X_\beta],
\]
which is precisely identity \eqref{RR} in Lemma \ref{Roberto}. This concludes
the proof of
the theorem.
\end{proof}

\begin{remark} \label{mamax}
Theorem \ref{teointduale} provides   an explicit integration of the second
equation in the Hamiltonian system \eqref{HJ}.
 The first Hamilton's equation in \eqref{HJ}, namely the equation $\dot\ga =
H_\xi(\gamma,\xi)$, reads
\[
 \dot\ga = - \sum_{j=1}^ r \langle \xi,X_j(\gamma)\rangle X_j(\gamma).
\]
By  \eqref{omegalambda} and \eqref{eqlaP}, we have
\[
 \langle  \xi,X_j(\gamma)\rangle = \la_j(\gamma(t) ) = P^v_j(\gamma(t))
\]
with $v=\lambda(0)$. By the formula \eqref{quadrato} for $X_j$, we 
obtain the following system of ordinary differential equations for   normal
extremals $\ga:[0,1]\to G =\Rn$
\[
 \dot\ga = - \sum_{j=1}^ r \sum_{j\prec\ell} 
\frac{(-1)^{|I(\ell)|} } {I(\ell)!} \gamma ^{I(\ell)} P^v_j(\gamma)  
\frac{\partial }{\partial x_{\ell}}.
\] 
\end{remark}

In order to characterize abnormal curves, we introduce the following algebraic
varieties.

\begin{defi} Let $G=\R^n$ be a free nilpotent Lie group of rank $r$.
 For any $v\in\Rn$, $v\neq0$, we call the set   
\[
  Z_v = \{x\in  \Rn: P_1^v(x) =\ldots=P_r^v(x) = 0\}
\]
an \emph{abnormal variety} of  $G$ \emph{of corank $1$}.

For linearly
independent vectors $v_1,\ldots, v_m\in\Rn$, $m\geq 2$, we call the set
$Z_{v_1}\cap \ldots \cap Z_{v_m}$ an  \emph{abnormal variety} of  $G$
\emph{of corank $m$}.
\end{defi}

Abnormal varieties depend on the system of coordinates of the second type
induced on $G$ by a 
Hall-Grayson-Grossman basis $X_1,\ldots ,X_n$ for the Lie algebra of $G$.

\begin{teo}\label{cor:polin} 
Let $G=\Rn$ be a free nilpotent Lie group and let $\ga:[0,1]\to G$
be a  horizontal curve with
$\ga(0)=0$. The following statements are
equivalent:

\begin{itemize}
 \item[A)] The curve $\ga$ is an abnormal extremal of corank $m\geq 1$.

 \item[B)] There exists $m$ linearly independent vectors
$v_1,\ldots,v_m\in\Rn$ such that $\gamma(t)\in  Z_{v_1}\cap\ldots\cap
Z_{v_m}$ for all $t\in[0,1]$.
\end{itemize}
\end{teo}

\begin{proof} 
Recall that the property of having corank $m$
for $\ga$ is equivalent to the existence of $m$ linearly independent solutions
to the system \eqref{8}.

Let  $\ga$ be an abnormal extremal and let 
$\la:[0,1]\to\Rn$ be a Lipschitz curve solving the system \eqref{8} and such
that
$\la\neq0$ pointwise. By Theorem \ref{teointduale}, there is $v\in\Rn$,
$v\neq0$, such that $\la_i = P_i^v(\ga)$ for all $i=1,\ldots,r$. For abnormal
extremals we have $\la_1=\ldots=\la_r=0$ on $[0,1]$, i.e., $\gamma(t)\in Z_v$
for all $t\in[0,1]$.   This shows that B) follows
from A).

On the other hand, if B) holds then the curve $\la$ defined in \eqref{eqlaP}
for any $v=v_1,\ldots,v_m$  satisfies the system \eqref{8} by Theorem
\ref{teointduale}, and moreover $\la_1=\ldots=\la_r=0$. From $v\neq 0$ it
follows that $\la(0) =v \neq 0$, and from the uniqueness of the solution to
\eqref{8} with initial condition it follows that $\la\neq0$ pointwise on
$[0,1]$. If $v_1,\ldots,v_m$ are linearly independent, then the corresponding
curves $\la$ are also linearly independent.
\end{proof}

\begin{remark}  
 Notice that when $v\neq0$ the zero set $Z_v$ is nontrivial, i.e., $Z_v\neq
\Rn$,   because by Proposition \ref{lemma-0} (ii) there is at least one index
$i=1,\ldots,r$ such that $P^v_i\neq0$ is not the
zero polynomial.
\end{remark}

The following corollary   easily follows from Definition \ref{GOG-EX}.

\begin{corol}\label{cor:polin:2} 
Let $G=\Rn$ be a free nilpotent Lie group with stratified algebra $\mathfrak
g = \mathfrak g_1\oplus\mathfrak g_2\oplus \cdots\oplus \mathfrak g_s$ and let
$r_1 = \mathrm{dim}( \mathfrak g_1)$ and $r_2 = \mathrm{dim}(\mathfrak g_2)$.
Let  $\ga:[0,1]\to G$
be a  horizontal curve such that
$\ga(0)=0$. The following statements are
equivalent:

\begin{itemize}
 \item[A)] The curve $\ga$ is a Goh extremal.

 \item[B)] There exists $v\in\Rn$, $v\neq 0$, such that for all
$i=1,\ldots,r_1+r_2$
and for all $t\in[0,1]$ there holds 
\begin{equation}\label{eqast}
    P_i^{v}(\ga(t))=0.
\end{equation}
\end{itemize}
\end{corol}

\begin{remark} \label{Dola} 
Given $v\in \R^n$ with $v\neq 0$, we call the set  
\[
   G_v =  \{x\in  \Rn: P_i^v(x) =0 \text{ for all }i=1,\ldots,r_1+r_2\}
\]
a {\em Goh variety} of the free nilpotent group $G$. By Proposition
\ref{lemma-0} (iii) we have $v_i=0$ for all $i=1,\dots,r_1+r_2$. Therefore, by
Proposition \ref{lemma-0} (ii) and Remark \ref{lemma-01}, there are at least
one index $i\in\{1,\ldots,r_1\}$ and one index $j\in\{r_1+1,\ldots,r_1+r_2\}$
such that
$P^v_i\neq 0$ and $P^v_j\neq0$.

The polynomials $P_i^v$ with $i=r_1+1,\ldots,r_1+ r_2$ completely describe
the horizontal sections of $G_v$. Namely, let $\gamma:[0,1]\to\Rn$ be a
\emph{horizontal} curve such that $\gamma(0) = 0$ and 
\[
 \gamma(t) \in  \{x\in  \Rn: P_i^v(x) =0 \text{ for all
}i=r_1+1,\ldots,r_1+r_2\}
\]
for all $t\in[0,1]$, where $v\in\Rn$ is a vector such that $v_i=0$, for all
$i=1,\ldots, r_1$. Then, by Proposition \ref{propderivpolinomi} for any index
$i=1,\ldots,r_1$ and for almost every  $t\in[0,1]$ we have
\[
 \frac{d}{dt } P_i^ v (\gamma(t)) = \sum_{k=1}^{r_1} \dot\gamma_k(t) X_k
P_i^v(\gamma(t)) =  \sum_{k=1}^{r_1} \dot\gamma_k(t) \sum_{\ell=1}^n
c_{ki}^\ell P_\ell^v(\gamma(t)) = 0,
\]
because the only indices $\ell$ involved in the previous sum are those with
degree 2. 
We deduce that $P_i^v(\gamma(t)) = P_i^v(\gamma(0)) = P_i^v(0) = v_i=0$.

\end{remark}

\section{Extremal curves in stratified groups}\label{sec:notfree}

In this section, we   give a partial 
classification of extremal curves in general stratified Lie groups. The
description is
not optimal because it involves a ``lifting'' procedure.   A more complete
result,
e.g., a purely algebraic characterization of abnormal
extremals, seems to require an explicit extension of Grayson-Grossman results
to nonfree nilpotent groups.

Let $Y_1,\ldots,Y_r$, $r\geq 2$, be smooth vector fields in $\Rn$ and let
$\pi:\Rn\to\R^m$, with $m\leq n$, be a smooth mapping.
Let us consider the following system of vector fields in $\R^m$:
\begin{equation}\label{xilox}
  X_1 =\pi_*(Y_1),\ldots, X_r= \pi_*(Y_r),                            
\end{equation}
where $\pi_*$ is the differential of $\pi$. Here and hereafter we assume that
$X_1,\dots, X_r$ are linearly independent at each point.

Let $\ga:[0,1]\to \R^m$ be a horizontal curve for $X_1,\ldots,X_r$ 
with $\ga(0)=0\in\R^m$. Then there are controls
$h_1,\ldots,h_r\in
L^2([0,1])$ such that
\begin{equation}\label{controllix}
\dot \ga(t) = \sum_{j=1}^r h_j(t) X_j(\ga(t)),\quad \textrm{for
a.e.~$t\in[0,1]$}.
\end{equation}
We call the curve $\kappa:[0,1]\to \Rn$ 
such that $\kappa(0)=0$ and 
$\dot \kappa (t) = \sum_{j=1}^r h_j(t) Y_j(\kappa(t))$
for a.e.~$t\in[0,1]$
the \emph{lift of $\ga$} to $\Rn$.

In the sequel, we denote by $T^*\Rn$ and $T^*\R^m$ the cotangent spaces to $\Rn$
and $\R^m$, respectively, and we denote by   $\pi^*:T^*
\R^m\to T^* \Rn$  the pull-back mapping induced by $\pi$. For the sake of
completeness, we provide the proof of the following easy fact.

\begin{prop} \label{progetto}
Let $\ga:[0,1]\to \R^ m $ be a horizontal curve for
$X_1,\ldots,X_r$ with $\gamma(0)=0$ and let $\kappa:[0,1]\to\R^m$ be the lift of
$\ga$ to $\Rn$. If $\xi:[0,1]\to T^*\R^m$ 
is a dual curve for $\gamma$ then $\pi^*\xi:[0,1]\to T^*\R^n$ is
a dual curve for $\kappa$.
\end{prop}

\begin{proof} 
 Let  $h_1,\ldots,h_r\in
L^2([0,1])$ be the controls of $\gamma$ as in \eqref{controllix} and
let $\Phi:[0,1]\times \R^m\to \R^m$ be the flow in $\R^m$ associated with these
controls.
Namely, let $\Phi(t,x)=\gamma_x(t)$ where $\gamma_x:[0,1]\to \R^m$ is the
solution to the problem
\[
 \dot\gamma_x =\sum_{j=1}^r h_j X_j(\gamma_x)\quad \text{a.e.~and
$\gamma_x(0)=x\in \R^m$}.
\] 
We also let $\Phi_t(x) = \Phi(t,x)$. 
Finally, let us denote by $\Psi:[0,1]\times \Rn\to \Rn$ the flow in $\Rn$
associated
with the same controls.

By the characterization \eqref{popix} of dual curves, we have 
\begin{equation}
\label{popix2}
 \xi (t) = (\Phi_t^{-1})^*\xi(0),
\end{equation}
for some $\xi(0)\in T_0^*\R^m$, $\xi(0)\neq0$. 
Above, $(\Phi_t^{-1})^*\xi(0)$
is the 
 pull-back via the diffeomorphism $\Phi_t^{-1}:\R^m\to \R^m$.

Let $\eta =\pi^*\xi$ be the pull-back to $\R^n$ of $\xi$ by  
 $\pi$. We claim that for any $t\in[0,1]$ we have
\begin{equation} \label{CLAM}
 \eta(t) = (\Psi_t^{-1})^*\eta(0),
\end{equation}
and thus $\eta:[0,1]\to T^*\R^n$ is a dual curve for   the
lift $\kappa$  to $\R^n$ of $\ga$.

To prove this claim, we preliminarily show that $\pi$ commutes
with the flows $\Phi$ and $\Psi$. Namely, for any $y\in \R^n$ and $t\in[0,1]$ we
have:
\begin{equation}\label{PUK}
   \pi( \Psi_t (y)) = \Phi_t( \pi(y)).
\end{equation}
This identity holds when $t=0$ because $\pi( \Psi_0 (y)) = \pi(y) =
\Phi_0(\pi(y))$.
Let us compute the derivatives in $t$ of the left and of the right  hand
sides in \eqref{PUK}.   We have:
\[\begin{split}
\frac{d}{dt}\pi( \Psi_t (y)) 
& 
=\pi_* \Big( \frac{d}{dt}  \Psi_t (y)\Big)
\ =\ 
\pi_* \Big( \sum_{j=1}^r h_j(t) Y_j( \Psi_t (y)) \Big) 
\\
& 
= \sum_{j=1}^r  h_j(t)\pi_*(Y_j( \Psi_t (y))) 
\ =\
 \sum_{j=1}^r  h_j(t)X_j( \pi(\Psi_t (y))).
\end{split}
\]
In the last line, we used \eqref{xilox}. On the other hand, we have
\[
\frac{d}{dt} \Phi_t (\pi(y)) 
= \sum_{j=1}^r  h_j(t)X_j( \Phi_t (\pi(y) )).
\]
It follows that the curves $t\mapsto   \pi( \Psi_t (y)) $ and $t\mapsto 
\Phi_t( \pi(y))$ solve the same differential equation with the same initial
condition. Now, identity \eqref{PUK} follows from the uniqueness of the
solution. The same argument  proves that
\begin{equation}\label{PUK1}
   \pi\circ  \Psi_t ^{-1} = \Phi_t^{-1} \circ  \pi,\quad t\in[0,1].
\end{equation}

Now we prove our main claim \eqref{CLAM}. By \eqref{popix}, by the composition
rule for the pull-back, and by \eqref{PUK1}, we have
\[
\begin{split}
\eta(t) &= \pi^*(\xi(t) ) = \pi^* ((\Phi_t^{-1})^*\xi(0)) 
\\
&=  (\Phi_t^{-1}
\circ \pi ) ^* \xi(0)=  (\pi\circ \Psi_t^{-1}
) ^* \xi(0)
\\&
=  ( \Psi_t^{-1})^* \pi^* (\xi(0))=   ( \Psi_t^{-1}) ^* \eta(0).
\end{split}
\] 
This ends the proof.
\end{proof}

\begin{remark} \label{kraft}
 The statement of Proposition \ref{progetto} can be improved in the following
sense: 
if $\gamma$ is an abnormal (respectively, Goh) extremal with dual curve $\xi$,
then
$\kappa$ is an abnormal (resp., Goh) extremal with dual curve $\pi^*\xi$.
\end{remark}

\begin{remark} \label{kraft1}
 Let $g$ be a quadratic form on the distribution $\mathcal D$ on $\R^m$ spanned
by $X_1,\ldots,X_r$, and let $h = \pi^* g$ be the pull-back of $g$ to the
distribution
$\mathcal E$ on $\R^n$ spanned by $Y_1,\ldots,Y_r$. Then, $\pi$ preserves the
length of horizontal curves. Hence, if $\gamma$ is a length minimizing curve in
$(\R^m,\mathcal D,g)$, then the lift
$\kappa$ of $\gamma$ to $\Rn$ is a length minimizing curve in $(\R^n,\mathcal
E, h)$.
\end{remark}

Now we pass to stratified Lie groups.
Let $G$ be a stratified $m$-dimensional Lie group of step $s\geq 2$ and rank
$r\geq 2$. Its Lie algebra $\mathfrak g$ admits a stratification 
\[
 \mathfrak g = \mathfrak g_{1} \oplus \ldots \oplus \mathfrak g_{s},
\]
where $\mathfrak g_{i+1} = [\mathfrak g_1,\mathfrak g_i]$ for   $i=1,\ldots,
s-1$ and
$\mathfrak g_i=\{0\}$ for $i>s$. The first layer $\mathfrak g_1$ has dimension
$r$. 

Let $F$ be the free nilpotent Lie group of step $s\geq 2$ and rank
$r\geq 2$. Let $\mathfrak f$ be its Lie algebra and let $n=\dim(\mathfrak f)$.
Let $Y_1,\ldots,Y_r$ be generators
of $\mathfrak f$.
There is a surjective Lie group homomorphism $\pi:F\to G$, with   differential
mapping $\pi_*:\mathfrak f\to\mathfrak g$, such that $\pi_\ast Y_i\in\mathfrak
g_1$, $i=1,\dots,r$. 
We complete $Y_1,\ldots,Y_r$ to a Hall basis
$Y_1,\ldots,Y_n$ of $\mathfrak f$.
By exponential coordinates of the second type, we can identify $F$ with $\Rn$
and assume that $Y_1,\ldots,Y_n$ is a Hall-Grayson-Grossman basis of left
invariant vector fields on $\Rn = F$.

Let $S\subset \{1,2,\ldots,n\}$ be a set such that the
vector fields $\pi_* Y_s$ with $s\in S$ form a basis for $\mathfrak g$. We have 
$m
= \# S$ and $S = \{s_1<s_2<\ldots<s_m\}$. We relabel the basis in the following
way
\[
  X_ i = \pi_*Y_{s_i},\quad i=1,\ldots,m.
\]
As $\mathfrak g$ has rank $r$, then the vector fields $X_i = \pi_*
Y_i\in\mathfrak g_1$,
 $i=1,\ldots,r$, are generators of $\mathfrak g$. 
There are constants $\zeta_{ij}\in\R$ such that
\begin{equation}\label{ZETA}
  \pi_* Y_i = \sum_{j=1}^m \zeta_{ij} X_j.
\end{equation}

In the following theorem, $G=\R^m$ is a stratified Lie group with rank $r$ and
step $s$, and $F=\Rn$ is the free nilpotent Lie group with rank $r$ and step
$s$.
The lift procedure from $G$ to $F$ is defined with respect to the systems of
vector fields fixed above.

\begin{teo} \label{PIV}
 Let $\ga:[0,1]\to G$ be a horizontal curve with $\gamma(0)=0$ and let
$\kappa:[0,1]\to F$ be the lift of $\gamma$ to $F$. If $\gamma$ is an extremal
curve with dual curve $\la:[0,1]\to\Rm$ then there is $v\in\Rn$ such that
\begin{equation}
 \label{pivo}
   v_i = \sum_{j=1}^m \zeta_{ij} v_{s_j},
\end{equation}
and
 the coordinates  of $\la$
satisfy, for all $i=1,\ldots,m$,
\[
           \la_i(t) = P^v_{s_i}(\kappa(t)), \quad t\in[0,1].
          \] 
\end{teo}

\begin{proof}
 If $\gamma$ is an extremal curve in $G$ with dual curve $\la$, then $\kappa$
is an extremal curve in $F$ with dual curve $\mu = \pi^*\la$, by Proposition
\ref{progetto}.
By Theorem \ref{teointduale}, there is $v\in\Rn$ such that the coordinates of
$\mu $ in the dual basis of $Y_1,\ldots, Y_n$ are
\[
  \mu_i(t) = P^v_i(\kappa(t)),\quad t\in[0,1],
\]
for any $i=1,\ldots,n$, and in fact we have $v = \mu(0) = \pi^*\la(0)$. By
\eqref{ZETA}, we have for $i=1,\ldots,n$
\[
 \mu_i= \mu(Y_i) = \pi^* \la (Y_i) =\la(\pi_* Y_i)=  \sum_{j=1}^m \zeta_{ij}
\la (X_j)=  \sum_{j=1}^m \zeta_{ij}
\mu_{s_j}.
\]
At $t=0$, this identity implies the relation \eqref{pivo} for $v$.
On the other hand, for any $i=1,\ldots,m$ we have
\[
 \la_ i = \la(X_i) = \la(\pi_* Y_{s_i}) =\mu (Y_{s_i}) =\mu_{s_i}=
P_{s_i}^v(\kappa).
\] 
\end{proof}

%%%%%
\section{Examples and applications}
\label{EXA}

In this section, we discuss some examples on how   theorems and formulae of
Sections \ref{agio} and \ref{sec:notfree} can be applied.

\subsection{Regularity of geodesics in stratified groups of step 3}
 
We give a short and alternative proof of a result proved in \cite{cinesi}.
 
\begin{teo}
Let $G$ be a stratified Lie group of step $3$ with a smooth left invariant
quadratic form $g$  on the horizontal distribution $\mathcal D$. Any length
minimizing curve in $(G,\mathcal D,g)$ is of class $C^\infty$. 
\end{teo}

\begin{proof} By Proposition
\ref{progetto} and Remark \ref{kraft1}, we can assume that $G$ is free.
Let $n$ and $r$ be the dimension and the rank of $G$,
respectively. By contradiction, assume there is a length minimizing
curve $\gamma:[0,1]\to G$ that is not of class $C^\infty$.  
Then $\gamma$ is a strictly abnormal extremal (normal extremals are of class
$C^\infty$) and thus a Goh extremal.

We can assume that $\gamma(0)=0$.
By Theorem \ref{cor:polin}, there are an index $i=1,\ldots,n$ with $d(i)=2$ and
$v\in\Rn$, $v\neq 0$, such that $P_i^v(\gamma)=0$. Notice that the polynomial
\[
    P_i^v(x) = -\sum_{\substack{j:d(j)=3\\\ell:d(\ell)=1}} c_{i\ell}^j
v_j x_\ell   \neq 0
\]
has homogeneous degree $1$. On differentiating in $t$ the identity
$P_i^v(\gamma)=0$, we obtain
\begin{equation}
                                  \label{DERO2}
\sum_{\substack{j:d(j)=3\\\ell:d(\ell)=1}} c_{i\ell}^j
v_j \dot\ga _\ell   =0.
                                 \end{equation}

Let $\mathfrak f$ be the (free) subalgebra of $\mathfrak g =\mathfrak
g_1\oplus \mathfrak g_2 \oplus \mathfrak g_3= \mathrm{Lie}(G)$ generated by
the subspace $\mathfrak f_1\subset\mathfrak g_1$ 
\[
  \mathfrak f_1 = \Big\{ \sum_{\ell:d(\ell)=1} x_\ell X_\ell \in \mathfrak g_1: 
\sum_{\substack{j:d(j)=3\\\ell:d(\ell)=1}} c_{i\ell}^j
v_j x_\ell  =0\Big\},
\]
where $X_1,\ldots,X_r\in\mathfrak g_1$ are Hall-Grayson-Grossman generators.
As $P_i^v(x)\neq 0$, we have $\mathrm{dim}(\mathfrak f_1) =r-1$.
Let $F\subset G$ be the stratified Lie group with Lie algebra $\mathfrak f$.
By \eqref{DERO2},  $\dot\ga$ is in $\mathfrak f$ and hence the curve $\ga$ is in
$F$. Moreover, 
 $\gamma$ is a length minimizer in $F$ for the restricted quadratic
form, and thus  $\gamma$ is a Goh extremal in $F$ (because it is not smooth).

We can repeat the above reduction argument to conclude that $\gamma$ is a Goh
extremal
in a free nilpotent group of rank $2$. Now the equation 
\[
    P_i^v(\gamma ) = -\sum_{\substack{j:d(j)=3 \\ \ell\in\{1,2\} }}
c_{i\ell}^j
v_j \gamma_\ell   =0,
\]
where $P_i^v$ is a nonzero polynomial, implies that $\gamma$ is a line, and
thus a smooth curve. This is a contradiction.
\end{proof}

\subsection{Regularity of generic length minimizing curves in stratified
Lie groups of rank 2}

We prove that in stratified Lie groups of rank 2 length minimizing curves are
``generically'' smooth. This is a special case of a deeper series of results by
W. Liu and H. J. Sussmann on \emph{regular abnormal extremals}, see \cite{LS}.

Let $G$ be a stratified Lie group with Lie algebra $\mathfrak g = \mathfrak
g_1\oplus\cdots\oplus \mathfrak g_s$ of rank $2$ and step $s\geq 3$.
We have $r_1 = \dim(\mathfrak g_1)=2$, 
$r_2 = \dim(\mathfrak g_2)=1$, and
$r_3 = \dim(\mathfrak g_3)\in\{1,2\}$.
We fix a Hall-Grayson-Grossman basis
$X_1,\ldots,X_n$, and we identify $G$ with $\Rn$. For a dual curve
$\la:[0,1]\to\Rn$, we let  $ \la^{(3)} = \la_{4}$ when $r_3=1$, and
 $ \la^{(3)} =( \la_{4},\la_5)$
when $r_3=2$.

\begin{prop} \label{pappa}
Let $\ga:[0,1]\to G=\Rn$ be a Goh extremal with dual curve $\la:[0,1]\to\Rn$.
Assume that $\la^{(3)}(t)\neq 0$ for all $t\in[0,1]$. Then $\ga$ is an analytic
curve.
\end{prop}

\begin{proof}
By Proposition \ref{progetto} and Remark \ref{kraft}, we can assume that $G$ is
free; in particular, $\mathrm{dim}(\mathfrak g_3)=2$. By Corollary
\ref{cor:polin:2},
there exists a $v\in\Rn$, $v\neq0$, such that for all $t\in[0,1]$ we have
\begin{equation}\label{LIPO}
 \gamma(t) \in \Sigma = \big\{ x\in\Rn :  P^v_3(x)=0\big\},
\end{equation}
 where $P_3^v(x)\neq0$ is a nonzero polynomial of the form \eqref{gestalt}.

By Proposition \ref{propderivpolinomi}, we have
\[
   X_1 P_3^v(x) = - P_4^v(x)\quad\text{and}\quad
   X_2 P_3^v(x) = - P_5^v(x).
\]
By Theorem \ref{teointduale}, $\la_4 = P_4^v(\ga)$ and $\la_5 = P_5^v(\ga)$,
and the assumption $\la^{(3)}\neq 0$ on $[0,1]$ implies that 
\[
   (X_1P_3^v(\ga))^2+(X_2P_3^v(\ga))^2\neq 0.
\]
In particular, $\nabla P_3^v \neq 0$ and thus $\Sigma$ is an analytic
hypersurface of $\Rn$ in a neighborhood of the support of $\gamma$.
Moreover, the distribution $\mathcal D(x) = \mathrm{span}\{
X_1(x),X_2(x)\big\}$ is transversal to $T_x\Sigma$:
\[
 \mathrm{dim}( T_x\Sigma \cap 
 \mathcal D(x)) = 1, \quad \textrm{for $x\in \gamma([0,1])$.}
\]
Inasmuch as $\gamma$ is horizontal and by \eqref{LIPO}, this implies that $\ga$
is
analytic.
\end{proof}

\begin{remark}
By Theorem 5 in \cite[p.~59]{LS}, Goh extremals as in Proposition
\ref{pappa}
are locally length minimizing.
 
\end{remark}

\subsection{A strictly abnormal curve}
We review  Gol\'e-Karidi's example \cite{GK} of a strictly
abnormal curve in a stratified group.

Let $G$ be a (free) nilpotent Lie group of rank $2$ and step $s\geq 4$, as in
the previous example. We fix a Hall basis $X_1,\ldots,X_n$ and we highlight the
first commutators
\begin{equation}\label{COMMU}
 X_3 = [X_2,X_1],\quad
 X_4 = [X_3,X_1],\quad
 X_5 = [X_3,X_2],\quad
 X_6 = %[X_3,X_1]= 
 [[X_3,X_1],X_1].
\end{equation}

We look for a Goh extremal $\gamma:[0,1]\to G=\Rn$. To this aim, let us
consider the polynomial
\[
  P_3^v (x) = \sum_{\alpha\in\I} \phi_{3\alpha}(v) x^\alpha = \sum_{\alpha\in\I}
\frac{(-1)^{|\alpha| }}{\alpha!} \sum_{j=1}^ n c_{3\alpha}^j v_j x^\alpha.
\]
We look for a $v\in\Rn$, $v\neq0$, and for a curve $\ga$ such that
$P_3^v(\gamma)=0$. The identities $P_1^v(\ga)=P_2^v(\gamma)=0$ will then follow
from Proposition \ref{propderivpolinomi}, by the argument of Remark \ref{Dola},
provided that we can choose $v_1=v_2=0$.

The coefficient  of $v_5$ in the polynomial $P_3^v(x)$ is
\[
  \sum_{\alpha\in\I} \frac{(-1)^{|\alpha|} }{\alpha!} c_{3\alpha}^5 
x^\alpha = -x_2,
\]
because the unique multi-index $\alpha\in\I$ such that $c_{3\alpha}^5\neq0$ is
$\alpha = (0,1,0,\ldots,0)$ and for this $\alpha$ we have $c_{3\alpha}^5= 1$,
see \eqref{COMMU}.

The coefficient  of $v_6$ in the polynomial $P_3^v(x)$ is
\[
  \sum_{\alpha\in\I} \frac{(-1)^{|\alpha|} }{\alpha!} c_{3\alpha}^6 
x^\alpha = \frac 12 x_1^2,
\]
because the unique multi-index $\alpha\in\I$ such that $c_{3\alpha}^6\neq0$ is
$\alpha = (2,0,\ldots,0)$ and for this $\alpha$ we have $c_{3\alpha}^6= 1$,
see again \eqref{COMMU}.

With $v_5=v_6=1$ and with $v_j=0$ otherwise, the polynomial $P_3^v(x)$ is
\[
 P_3^v(x) = \frac{1}{2} x_1^2 - x_2.
\]
Now we clearly see that the horizontal curve $\gamma:\R \to\Rn$ such that
$\gamma_1(t) = t$ and $\gamma_2(t) = \frac 12 t^2$ is a Goh extremal, by
Corollary  \ref{cor:polin:2} and Proposition \ref{propderivpolinomi}.

We prove that $\ga$ is a strictly abnormal curve. If $\ga$ were a normal
extremal with dual curve $\la$, then by \eqref{acca} and \eqref{omegalambda} we
would have
\[
 \dot\gamma_ j= h_ j =-\langle \xi,X_j\rangle = \lambda_j,\quad j=1,2. 
\]
This yields $\lambda_1 = 1$ and $\lambda_2 = t$.
Moreover, the system of equations \eqref{8} is
\[
   \dot\la_ i =-\sum_{k=1}^n \big\{ c_{i1}^ k + c_{i2}^k t\big\}
\lambda_k,\quad i=1,\ldots,n.
\]
When $i=1$ we obtain $ t\la_3=0$, whereas for  $i=2$ we 
obtain $\la_3=-1$. This is a contradiction.

\begin{remark}
As proved in \cite{GK}, the curve $\gamma$ above is locally length minimizing.
Indeed, with the previous choice of $v$, we have by Proposition
\ref{propderivpolinomi} that
\[
\lambda_5(t)=P_5^v(\gamma(t))=-(X_2 P_3^v)(\gamma(t))=1\,.
\]
Therefore, $\gamma$ satisfies the assumptions in Proposition \ref{pappa} and is
locally length minimizing by \cite{LS}.
\end{remark}
    
Let us give here the complete formula for the polynomial $P_3^v(x)$ when $G$ is
a free nilpotent Lie group of rank 2 and step $s=6$. This group is
diffeomorphic to $\R^{23}$. The polynomial is
 \[
\begin{split}
P_{3}^v(x)  =\  & v_3 - v_4 x_1 - v_5 x_2
+v_6  \frac{x_1^{2}}{2}+ v_7 x_1 x_2 +v_8  \frac{x_2^{2}}{2}
-v_9  \frac{x_1^{3}}{6} - v_{10}  \frac{x_1^{2}x_2}{2} 
- v_{11}  \frac{x_1x_2^{2}}{2} 
- v_{12}  \frac{x_2^{3}}{6}
 \\
&~~
 + v_{13}(x_4+x_1 x_3) +    v_{14} (x_5+x_2 x_3 )
  \\
&~~
 +v_{15} \frac{x_1^{4}}{24} 
  +v_{16} \frac{x_1^{3}x_2}{6} 
   +v_{17} \frac{x_1^{2}x_2^2}{4} 
    +v_{18} \frac{x_1x_3^{2}}{6} 
    +v_{19} \frac{x_2^{4}}{24} 
     \\
&~~
 +   v_{20}( x_6-\frac{x_1^2x_3}{2} )+   v_{21}(x_7-x_1x_2x_3)
 +  v_{22} ( x_8- \frac{x_2^2x_3}{2} )+   v_{23}(x_2x_4-x_1x_5)
 .
\end{split}
\] 
The homogeneous degree of the polynomial is at most $4$. The variables
$x_9,\ldots, x_{23}$
do not appear.

\subsection{Rank 3 and step 4}\label{Example2}
Let $G$ be the free nilpotent Lie group of rank 3 and step 4. This group is
diffeomorphic to $\R^{32}$. By 
 Corollary \ref{cor:polin:2} and Remark \ref{Dola}, Goh
extremals of $G$ starting from $0$ are
precisely the horizontal curves $\gamma$ contained in the algebraic
set
\[
 \Sigma = \big  \{ x\in\R^{32}: P_4^v(x) = P_5^v(x)=P_6^v(x) =0 \big\}, 
\]
for some $v\in \R^{32}$ such that $v_1=\ldots=v_6=0$.
We list the  the polynomials defining this algebraic set:
\[
\begin{split}
 P_{4}^v(x)  & =   -x_1 v_7-x_2 v_8 
-x_3 v_9 +x_5v_{30}
+x_6v_{31}
\\
&\qquad 
+ \frac{x_1^{2}}{2} v_{15}
+x_1 x_2 v_{16} 
+x_1 x_3 v_{17} 
+ \frac{x_2^{2}}{2} v_{18} +x_2 x_3 v_{19}
+ \frac{x_3^{2}}{2} v_{20}
\\
P_{5}^v(x)  & =  -x_1 v_{10}-x_2 v_{11} 
-x_3 v_{12} 
-x_4v_{30}
+x_6v_{32}
\\
&\qquad
+ \frac{x_1^{2}}{2} v_{21}
+x_1 x_2 v_{22} 
+x_1 x_3 v_{23} 
+ \frac{x_2^{2}}{2} v_{24} +x_2 x_3 v_{25}
+ \frac{x_3^{2}}{2} v_{26}
\\
P_{6}^v(x)  & =  x_1 (v_9
- v_{11})
-x_2 v_{13} 
-x_3 v_{14} 
-x_4 v_{31} -x_5v_{32}
%+x_6v_{31}
+x_1^{2}(- \frac{1}{2} v_{17}
+\frac{1}{2} v_{22}
+ v_{30})
\\
&\qquad
+x_1 x_2 (-v_{19} 
+  v_{24}
+  v_{31})
+x_1 x_3 (-v_{20} 
+ v_{25})
+ \frac{x_3^{2}}{2} v_{29}
.
\end{split}
\] 
The set $\Sigma$ is an intersection of quadrics.

When $v_7=1$ $v_{18}=2$ and $v_j=0$ otherwise, we have $P^v_4(x) =  
x_2^2-x_1$, $P_5^v(x)=P_6^v(x)=0$.
Let $\phi:[0,1]\to\R$ be \emph{any} Lipschitz function with $\phi(0)=0$. The
horizontal curve
$\gamma:[0,1]\to\R^{32}$ such that $\gamma(0)=0$, $\gamma_1(t) = t^2$,
$\gamma_2(t)=t$ and $\gamma_3(t) = \phi(t)$ is a Goh extremal with purely
Lipschitz regularity.

An interesting question raised by the referee is whether the curve $\gamma$
constructed above is length minimizing (for any choice of the metric in the
horizontal bundle). In its generality, the question is open. We can only
give the following partial answer in the negative, related to the regularity of
$\phi$.

Assume that there exists a point
$t_0\in(0,1)$   such that the following limits exist and are different 
\begin{equation}\label{Lpm}
  \lim_{t\to0^+} \frac{\phi(t_0+t) -\phi(t_0)}{t} \neq                   
                             \lim_{t\to 0^-} \frac{\phi(t_0+t)
-\phi(t_0)}{t}.
\end{equation}
Then we claim that $\gamma$ is not length minimizing (for any
choice of the metric in the
horizontal bundle).

We sketch the argument. We can assume that $\gamma(t_0)=0$. We perform a blow-up
of the curve $\gamma$, as in
\cite[Section 2]{LM}. The  curve obtained by the blow-up of $\gamma$ consists
of two half-lines emanating from $0$ and forming a corner, by \eqref{Lpm}. This
curve will be contained in a Carnot group of rank 2  and step at most 4 (see
Remark 2.5 in \cite{LM}). The length minimality of $\gamma$ would imply the
length
minimality of the limit curve. However, by Example 4.6 in \cite{LM},
length minimizing curves  in Carnot groups of rank 2 and step at most 4 are
$C^\infty$ smooth. Thus \eqref{Lpm} prevents $\gamma$ to be length minimizing.


\begin{thebibliography}{99}
\bibitem{AS} {\sc A.~Agrachev \& Y. L. Sachkov}, {\em Control Theory from the
Geometric Viewpoint}. Encyclopaedia of Mathematical Sciences, 87. Control Theory
and Optimization, II. Springer-Verlag, Berlin, 2004. xiv+412 pp.

\bibitem{AS2} {\sc A.~Agrachev, A.~Sarychev}, {\em 
Abnormal sub-Riemannian geodesics: Morse index and rigidity}. Ann.
Inst. Henri Poincar\'e, {\bf 13} n. 16 (1996), 635--690.
 
\bibitem{Bryant-Hsu} {\sc R.~L.~ Bryant \& L.~Hsu}, {\em 
Rigidity of integral curves of rank 2 distributions}.
Invent. Math. 114 (1993), no. 2, 435--461. 

\bibitem{CJT} {\sc Y. Chitour, F. Jean \& E. Tr\' elat}, {\em Genericity
results for singular curves}.  J. Differential Geom.  73  (2006),  no. 1,
45--73.


\bibitem{GK} {\sc C. Gol\'e \& R. Karidi},  {\em A note on Carnot geodesics
in nilpotent Lie groups}.  J. Dynam. Control Systems  1  (1995),  no. 4,
535--549. 

\bibitem{GrGr} {\sc M. Grayson \& R. Grossman}, {\em Models for free nilpotent
Lie algebras}. J. Algebra 135 (1990), no. 1, 177--191.

\bibitem{Hall} {\sc M. Hall Jr.}, {\em A basis for free Lie rings and higher
commutators in free groups}. Proc. Amer. Math. Soc. 1, (1950), 575--581.
 

\bibitem{LM} {\sc G. P. Leonardi \& R. Monti}, {\em End-point equations and
regularity of sub-Riemannian geodesics}. Geom. Funct. Anal. {\bf 18} (2008), no.
2, 552--582.
 

\bibitem{montgomery} {\sc R. Montgomery}, {\em A tour of Sub-Riemannian
Geometries, Their Geodesics and Applications}, AMS, 2002.

\bibitem{montg1} {\sc 
R.~Montgomery}, {\em Abnormal minimizers}, SIAM J. Control Optim., {\bf
32} (1994), 1605--1620.


\bibitem{monti} {\sc R. Monti}, Regularity results for sub-Riemannian
geodesics, Preprint 2012.

\bibitem{LS} {\sc W. Liu \& H. J. Sussmann}, {\em Shortest paths for
sub-Riemannian metrics on rank-two distributions.} Mem. Amer. Math. Soc. {\bf
118} (1995), no. 564, x+104 pp.

 

\bibitem{cinesi} {\sc K. Tan \& X. Yang}, {\em Subriemannian geodesics of Carnot
groups of step 3}.
\href{http://arxiv.org/pdf/1105.0844v1.pdf}{http://arxiv.org/pdf/1105.0844v1.pdf
}.



\bibitem{ZZ} {\sc I.~Zelenko \& M.~Zhitomirski\u{\i}}, {\em 
Rigid paths of generic 2-distributions on 3-manifolds}.
Duke Math. J. 79 (1995), no. 2, 281--307. 

 

\end{thebibliography}
\end{document}